\documentclass[12pt, reqno]{amsart}

\usepackage[a4paper, centering, total={170mm,240mm}]{geometry}

\usepackage{amssymb,latexsym}
\usepackage{blindtext}
\usepackage[backgroundcolor=lightgray, linecolor=olive, textsize=tiny]{todonotes}
\usepackage{esint,dsfont}
\usepackage{MnSymbol}

\usepackage[colorlinks=true, linkcolor=blue, citecolor=purple, urlcolor=blue]{hyperref}

\def\R{{\mathbb R}}
\def\N{\mathbb{N}}
\def\C{\mathbb{C}}

\DeclareMathOperator{\Id}{Id}
\def\op{\mathrm{op}}

\newtheorem{prop}{\bf Proposition}[section]
\newtheorem{thm}[prop]{\bf Theorem}
\newtheorem{cor}[prop]{\bf Corollary}
\newtheorem{lem}[prop]{\bf Lemma}
\newtheorem{rmk}[prop]{\it Remark}

\newtheorem{defi}[prop]{\bf Definition}
\newtheorem{ques}[prop]{\bf Question}

\begin{document}

\title[Relative property (T) for uniformly bounded representations]{Relative property (T) for uniformly bounded representations}

\author{Guillaume Dumas}
\address{Department of Mathematics, University of Maryland, College Park,
		4176 Campus Drive, College Park, MD 20742}
\email{gdumas@umd.edu}

\author{Ignacio Vergara}
\address{Departamento de Matem\'atica y Estad\'istica, Universidad de La Frontera, Avenida Francisco Salazar 01145, Temuco 4811230, Chile}
\email{ign.vergara.s@gmail.com}

\thanks{I.V. was supported by the ANID-SIA project 85250079 and the ECOS project 230003.}


\begin{abstract}
For pairs $(G,H)$, where $G$ is a locally compact group and $H$ is a closed subgroup of $G$, we introduce analogous versions of the full group $\mathrm{C}^*$-algebra and the Fourier--Stieltjes algebra for uniformly bounded representations. We use these objects to characterise relative property (T) in this more general setting in terms of Kazhdan projections and invariant means. We also show that, for semidirect products of the form $G=H\ltimes N$, where $N$ is a nilpotent group, relative property (T) for the pair $(G,N)$ is equivalent to its uniformly bounded version. We first prove this result when $N$ is abelian, and then proceed by studying the behaviour of relative property (T) under quotients by central subgroups, thereby extending a theorem of Serre to the uniformly bounded setting.
\end{abstract}

\maketitle

\thispagestyle{empty}

\section{Introduction}

Property (T) was introduced by Kazhdan \cite{Kaz} as a rigidity phenomenon for unitary representations of locally compact groups. In that same paper, a relative version of property (T) for pairs of groups was already implicit and it was used to prove that $\operatorname{SL}(3,\mathbb{K})$ has property (T) by studying the behaviour of its representations with respect to certain subgroups. Nowadays, relative property (T) has been widely studied, and its relevance has been made much more explicit; see e.g. \cite{Cor, Jau}.

In more recent years, the consideration of stronger versions of property (T) has gained increasing interest. One way of generalising property (T) is by considering uniformly bounded representations of a group and studying the rigidity properties of this more general class; see  \cite{DruNow, Mim, Ver}. 

The goal of this paper is to define and study a relative version of property (T) for uniformly bounded representations.

Let $G$ be a locally compact group. For $c\geq 1$, let $\mathcal{R}_c(G)$ denote the class of (SOT-continuous) uniformly bounded representations $\pi:G\to\mathbf{B}(\mathcal{H}_\pi)$, where $\mathcal{H}_\pi$ is a Hilbert space, and
\begin{align*}
	|\pi|=\sup_{s\in G}\|\pi(s)\|\leq c.
\end{align*}
 The group $G$ has property (T)${}_c$ if, for any $\pi\in \mathcal{R}_c(G)$, the existence of almost invariant vectors implies the existence of a nonzero invariant vector (see \cite{Ver} for a detailed study of these properties).

Let $H$ be a closed subgroup of $G$. We can define relative properties in a similar fashion.
\begin{defi}\label{def:relT}
    Let $c\geq 1$. The pair $(G,H)$ has relative property (T)${}_c$ if for any $\pi\in \mathcal{R}_c(G)$ with almost invariant vectors, $\pi\vert_H$ has a nonzero invariant vector.

    The pair $(G,H)$ has relative property (T${}_{ub}$) if $(G,H)$ has relative (T)${}_c$ for any $c\geq 1$.
\end{defi}

In \cite{Ver}, property (T)${}_c$ was characterised in terms of Kazhdan projections and invariant means. Our first result gives a similar characterisation for relative property (T)${}_c$.

For every $\pi\in\mathcal{R}_c(G)$, we let $\mathcal{H}_\pi^G$ denote the subspace of $G$-invariant vectors of $\mathcal{H}_\pi$. If $H$ is a closed subgroup of $G$, we can restrict $\pi$ to $H$ and define $\mathcal{H}_\pi^H$ similarly. Observe that $\mathcal{H}_\pi^G\subseteq\mathcal{H}_\pi^H$. We let $\mathcal{R}_c(G,H)$ denote the subclass of $\mathcal{R}_c(G)$ given by all representations $\pi$ such that
\begin{align*}
    \mathcal{H}_\pi^G=\mathcal{H}_\pi^H.
\end{align*}
For every $\pi\in\mathcal{R}_c(G,H)$, we can decompose 
\begin{align*}
	\mathcal{H}_\pi=\mathcal{H}_\pi^G\oplus \left(\mathcal{H}_{\pi^*}^G\right)^\perp=\mathcal{H}_\pi^H\oplus \left(\mathcal{H}_{\pi^*}^H\right)^\perp,
\end{align*}
where $\pi^*$ denotes the adjoint representation $\pi^*(s)=\pi(s^{-1})^*$, and $\left(\mathcal{H}_{\pi^*}^G\right)^\perp$ is the orthogonal complement of the subspace of $\pi^*$-invariant vectors $\mathcal{H}_{\pi^*}^G$.

We define the algebra $\tilde{A}_c(G,H)$ as the completion of $L^1(G)$ (after modding out by the null space) for the norm
\begin{align*}
\|f\|_{\tilde{A}_c(G,H)}=\sup_{\pi\in\mathcal{R}_c(G,H)}\|\pi(f)\|.
\end{align*}
We also let $B_c(G,H)$ denote the dual Banach space of $\tilde{A}_c(G,H)$. This is a function space, given by matrix coefficients of representations of $G$. Moreover, it always carries a unique invariant mean; see Section \ref{Ssec_matrix_coef}. We obtain the following characterisation.

\begin{thm}\label{Thm_char_(T)c_intro}
	Let $G$ be a locally compact group, $H$ a closed subgroup of $G$, and $c\geq 1$. The following are equivalent:
	\begin{itemize}
		\item[(i)] The pair $(G,H)$ has relative property (T)${}_c$.
		\item[(ii)] There exists an idempotent $P\in \tilde{A}_c(G,H)$ such that
		\begin{itemize}
			\item for every $\pi\in\mathcal{R}_c(G,H)$, $\pi(P)$ is the projection onto $\mathcal{H}_\pi^G$ along $\left(\mathcal{H}_{\pi^*}^G\right)^\perp$,
			\item there exists a sequence of continuous, compactly supported probability measures $(\rho_n)$ on $G$ such that
			\begin{align*}
				\lim_{n\to\infty}\|\rho_n-P\|_{\tilde{A}_c(G,H)}=0.
			\end{align*}
		\end{itemize}
		\item[(iii)] The unique invariant mean $m\in B_c(G,H)^*$ is weak*-continuous. In other words, it belongs to $\tilde{A}_c(G,H)$.
	\end{itemize}
	Moreover, in this case, $P=m$.
\end{thm}

When the pair $(G,H)$ has relative property (T)${}_c$, we will refer to the idempotent $P\in \tilde{A}_c(G,H)$ given by Theorem \ref{Thm_char_(T)c_intro} as the Kazhdan projection of $\tilde{A}_c(G,H)$.

\begin{rmk}
	If $H=G$, then relative property (T)${}_c$ for the pair $(G,G)$ is exactly property (T)${}_c$ for the group $G$, and therefore \cite[Theorem 1.1]{Ver} becomes a particular case of Theorem \ref{Thm_char_(T)c_intro}.
\end{rmk}

\begin{rmk}
	Theorem \ref{Thm_char_(T)c_intro} seems to be new even for $c=1$. In that case, $\tilde{A}_1(G,H)$ is a quotient of the full group $\mathrm{C}^*$-algebra $C^*(G)$, and $B_1(G,H)$ is a subspace of the Fourier--Stieltjes algebra $B(G)$. This says that relative property (T) admits characterisations in terms of Kazhdan projections and invariant means as above, extending results from \cite{AkeWal, HaKndL, Val} to the relative setting.
\end{rmk}

The characterisation of property (T) as the weak*-continuity of the unique invariant mean on $B(G)$ was first observed by Haagerup, Knudby and de Laat; see \cite[\S 4]{HaKndL}. This was their starting point for introducing a stronger version of property (T)--called it property (T${}^*$)--by requiring the unique invariant mean on the space of completely bounded multipliers of the Fourier algebra $M_0A(G)$ to be weak*-continuous; see \cite[\S 5]{HaKndL} for details. As observed in \cite[\S 1]{Ver}, property (T${}^*$) implies property (T)${}_c$ for every $c\geq 1$ because the inclusion $B_c(G)\hookrightarrow M_0A(G)$ is weak*-weak*-continuous. In other words, property (T${}^*$) implies property (T${}_{ub}$), and it is not known whether the converse holds; see \cite[Question 1]{Ver}. This raises the following question.

\begin{ques}\label{Ques_rel_T*}
	Can we define a relative version of property (T${}^*$) for pairs $(G,H)$ as the weak*-continuity of the invariant mean on some space of multipliers? This property should satisfy:
	\begin{itemize}
		\item relative property (T${}^*$) implies relative property (T${}_{ub}$),
		\item relative property (T${}^*$) for $(G,G)$ coincides with property (T${}^*$) for $G$.
	\end{itemize}
\end{ques}

The following result is well known in the classical setting ($c=1$). More precisely, if $G$ is a group with property (T), and $H$ is a normal subgroup of $G$, then it follows from the definitions and basic stability properties that $G/H$ has property (T), and the pair $(G,H)$ has relative property (T). Moreover, the converse is also true; see e.g. \cite[Proposition 2.2.9]{Jau} or \cite[Proposition 2.4.4]{Cor}. This characterisation can be extended to $c\geq 1$ by similar arguments. We give an alternative proof of this fact using the language of Kazhdan projections and invariant means.

\begin{prop}\label{Prop_G/H_(G,H)}
	Let $G$ be a locally compact group, $H$ a closed, normal subgroup of $G$, and $c\geq 1$. Then $G$ has property (T)${}_c$ if and only if the following two conditions hold: 
	\begin{itemize}
		\item The group $G/H$ has property (T)${}_c$.
		\item The pair $(G,H)$ has relative property (T)${}_c$.
	\end{itemize}
\end{prop}

The proof of Proposition \ref{Prop_G/H_(G,H)} that we present here is also motivated by Question \ref{Ques_rel_T*}, as we expect that, if one can give a good definition of relative property (T${}^*$), then it should also satisfy a similar characterisation.

We also study property (T)${}_c$ for some families of pairs where $G$ is a semidirect product.

Our second result concerns abelian groups. Let $A$ be a locally compact abelian group and $G$ a semidirect product of $A$ by another locally compact group $H$. It is well known that the usual unitary relative property (T) for the pair $(G,A)$ can be linked with invariant means on the Pontryagin dual $\hat{A}$ of $A$. It was proven by Shalom \cite[Theorem 5.5]{Sha} that if the only $H$-invariant mean on $\hat{A}$ is the Dirac $\delta_1$, then $(G,A)$ has relative property $(T)$. In \cite{CorTes, Ioa}, a complete characterisation of relative property $(T)$ for $(G,A)$ is given in terms of $H$-invariant means on $\hat{A}$ (see Theorem \ref{thm:cortes} for the precise statement). We extend this characterisation to the uniformly bounded case and obtain the following theorem.

\begin{thm}\label{thm:semidirectab}
     Let $G=H\ltimes A$ be a locally compact group with $A$ closed abelian normal subgroup. Then $(G,A)$ has relative property (T) if and only if $(G,A)$ has relative property (T${}_{ub}$).
\end{thm}

One of the key points is that, when restricted to $A$, uniformly bounded representations are unitarisable as $A$ is amenable. However, this is far from true on $G$ itself, forcing the use of \textit{complex} measures instead of merely probability measures.

We can in fact extend this result by replacing $A$ with a locally compact nilpotent group $N$. To achieve this, we adapt ideas from \cite{CWMS} to the context of uniformly bounded representations and prove the following criterion, extending \cite[Theorem 1.5]{CWMS} to $c>1$. Recall that $[H,H]$ denotes the commutator subgroup of $H$.

\begin{thm}\label{thm:quotab_intro}
    Let $G$ be a second countable locally compact group, $H$ a closed normal subgroup of $G$, $A$ a closed abelian subgroup of $H$, normal in $G$ and $c\geq 1$. If $(G/A,H/A)$ has relative property (T)${}_{c^2}$ and every $H$-invariant finite measure on $\hat{A}$ is supported on the set of fixed points of the action of $H$, then $(G,\overline{[H,H]})$ has relative property (T)${}_c$.
\end{thm}

This result, together with Theorem \ref{thm:semidirectab}, allows us to obtain the following.

\begin{thm}\label{thm:mainthmnil}
     Let $G=H\ltimes N$ be a locally compact second countable group with $N$ closed nilpotent normal subgroup. Then $(G,N)$ has relative property $(T)$ if and only if $(G,N)$ has relative property (T${}_{ub}$).
\end{thm}

As a by-product, we obtain the following generalisation of a theorem of Serre \cite[Theorem 1.7.11]{BHV} to the uniformly bounded setting. Recall that $Z(G)$ denotes the centre of $G$.

\begin{cor}\label{Cor_Serre_intro}
    Let $G$ be a locally compact second countable group, $A\leq Z(G)$ a closed subgroup, and $c\geq 1$. If $G/A$ has property (T)${}_{c^2}$ and $G^{ab}=G/\overline{[G,G]}$ is compact, then $G$ has property (T)${}_c$.
\end{cor}

This paper is organised as follows. In Section \ref{sec:prelim}, we gather all the preliminaries needed in the proofs of our main results. Section \ref{sec:charct} is devoted to relative property (T)${}_c$ and the proofs of Theorem \ref{Thm_char_(T)c_intro} and Proposition \ref{Prop_G/H_(G,H)}. In Section \ref{sec:ab}, we focus on pairs of the form $(H\ltimes A, A)$, where $A$ is abelian, and we prove Theorem \ref{thm:semidirectab}. Finally, in Section \ref{sec:nil}, we extend these ideas to nilpotent groups, and prove Theorem \ref{thm:quotab_intro}, Theorem \ref{thm:mainthmnil}, and Corollary \ref{Cor_Serre_intro}.

\section{Preliminaries}\label{sec:prelim}

In this section, we gather some preliminaries that will be needed throughout the paper.

\subsection{Relative property \texorpdfstring{(T)${}_c$}{(T)c} and Kazhdan pairs}
We begin by discussing relative property (T)${}_c$ in more detail. Let $G$ be a locally compact group, $c\geq 1$, and $\pi\in\mathcal{R}_c(G)$. Recall that $\mathcal{H}_\pi^G$ denotes the subspace of $G$-invariant vectors of $\mathcal{H}_\pi$. We have the following direct sum decomposition:
\begin{align}\label{dec_G-inv}
	\mathcal{H}_\pi=\mathcal{H}_\pi^G\oplus \left(\mathcal{H}_{\pi^*}^G\right)^\perp,
\end{align}
where $\pi^*$ denotes the adjoint representation $\pi^*(s)=\pi(s^{-1})^*$, and $\left(\mathcal{H}_{\pi^*}^G\right)^\perp$ is the orthogonal complement of the subspace of $\pi^*$-invariant vectors $\mathcal{H}_{\pi^*}^G$; see \cite[\S 2.3]{DruNow}. Moreover, in the decomposition \eqref{dec_G-inv}, both subspaces are $G$-invariant. 

Let now $H$ be a closed subgroup of $G$. By restricting $\pi$ to $H$, we get
\begin{align}\label{dec_H-inv}
	\mathcal{H}_\pi=\mathcal{H}_\pi^H\oplus \left(\mathcal{H}_{\pi^*}^H\right)^\perp,
\end{align}
as in \eqref{dec_G-inv}. Observe that, when $H$ is a normal subgroup, these subspaces are $G$-invariant too. Let $Q$ be a compact subset of $G$, and $\varepsilon>0$. We say that $(Q,\varepsilon)$ is a Kazhdan pair for $(\mathcal{R}_c(G),H)$ if, for every $\pi\in\mathcal{R}_c(G)$ with $\mathcal{H}_{\pi}^H=\{0\}$ and every $\xi\in \mathcal{H}_{\pi}$,
\begin{align}\label{spec_gap}
	\max_{s\in Q}\|\pi(s)\xi-\xi\|\geq \varepsilon\|\xi\|.
\end{align}
We will say that the pair $(G,H)$ has property (T)${}_c$ if $(\mathcal{R}_c(G),H)$ admits a Kazhdan pair. If, in addition, we assume that $G$ is compactly generated, then this condition is equivalent to the fact that, for every compact generating subset $Q\subseteq G$, there is $\varepsilon>0$ such that $(Q,\varepsilon)$ is a Kazhdan pair for $(\mathcal{R}_c(G),H)$; see e.g. \cite[Proposition 1.1.6]{Jau}.

Observe that, when $H=G$, we recover the definition of property (T)${}_c$ from \cite{Ver}.

\subsection{The class of representations \texorpdfstring{$\mathcal{R}_c(G,H)$}{Rc(G,H)}}
We now explore in more detail the definition of $\mathcal{R}_c(G,H)$, as well as the reasons behind this choice of subclass of representations. Let $G$ be a group with property (T)${}_c$, and let $\pi\in\mathcal{R}_c(G)$. Then the decomposition \eqref{dec_G-inv} is given by a subspace on which the $G$-action is trivial ($\mathcal{H}_\pi^G$) and another one that witnesses the spectral gap condition ($\left(\mathcal{H}_{\pi^*}^G\right)^\perp$). Moreover, property (T)${}_c$ can be characterised in terms of the projection onto $\mathcal{H}_\pi^G$ along $\left(\mathcal{H}_{\pi^*}^G\right)^\perp$; see \cite[Corollary 3.4]{Ver}.

Let now $H$ be a closed subgroup of $G$. The definition of relative property (T)${}_c$ relates the subspace on which the $H$-action is trivial ($\mathcal{H}_\pi^H$) with the subspace that witnesses the spectral gap condition ($\left(\mathcal{H}_{\pi^*}^G\right)^\perp$). Hence, our characterisation should be expressed in terms of a projection onto $\mathcal{H}_\pi^H$ along $\left(\mathcal{H}_{\pi^*}^G\right)^\perp$. In other words, we would like to have a decomposition of the form
\begin{align*}
	\mathcal{H}_\pi=\mathcal{H}_\pi^H\oplus \left(\mathcal{H}_{\pi^*}^G\right)^\perp.
\end{align*}
From the decompositions \eqref{dec_G-inv}, \eqref{dec_H-inv}, and the inclusions $\mathcal{H}_\pi^G\subseteq \mathcal{H}_\pi^H$, $\left(\mathcal{H}_{\pi^*}^H\right)^\perp\subseteq\left(\mathcal{H}_{\pi^*}^G\right)^\perp$, we see that this can only happen if $\mathcal{H}_\pi^G=\mathcal{H}_\pi^H$, which in turn implies that $\left(\mathcal{H}_{\pi^*}^G\right)^\perp=\left(\mathcal{H}_{\pi^*}^H\right)^\perp$.

Therefore, we define $\mathcal{R}_c(G,H)$ as the subclass of $\mathcal{R}_c(G)$ consisting of all representations $\pi\in\mathcal{R}_c(G)$ such that $\mathcal{H}_\pi^G=\mathcal{H}_\pi^H$. In other words, for $\pi\in\mathcal{R}_c(G,H)$, the decompositions \eqref{dec_G-inv} and \eqref{dec_H-inv} coincide. Observe also that $\mathcal{R}_c(G,G)=\mathcal{R}_c(G)$.

\subsection{Spaces of matrix coefficients}\label{Ssec_matrix_coef}

For a locally compact group $G$, let $B_c(G)$ denote the space of matrix coefficients of representations in $\mathcal{R}_c(G)$. More precisely, for every $\varphi\in B_c(G)$, there are $\pi\in\mathcal{R}_c(G)$ and $\xi,\eta\in\mathcal{H}_\pi$ such that, for all $s\in G$,
\begin{align}\label{phi_mcoef}
	\varphi(s)=\langle\pi(s)\xi,\eta\rangle.
\end{align}
We endow $B_c(G)$ with the norm
\begin{align*}
	\|\varphi\|_{B_c(G)}=\inf \|\xi\| \|\eta\|,
\end{align*}
where the infimum is taken over all decompositions as in \eqref{phi_mcoef}. As shown in \cite[Proposition 2.10]{Ver}, $B_c(G)$ is a dual Banach space. We define $\tilde{A}_c(G)$ as the completion of $L^1(G)$ for the norm
\begin{align*}
	\|f\|_{\tilde{A}_c(G)}=\sup_{\pi\in\mathcal{R}_c(G)}\|\pi(f)\|,
\end{align*}
where
\begin{align*}
	\pi(f)=\int_Gf(s)\pi(s)\,ds\ \in\ \mathbf{B}(\mathcal{H}_\pi).
\end{align*}
Then $B_c(G)$ can be identified with the dual space of $\tilde{A}_c(G)$ for the usual $(L^\infty,L^1)$ duality pairing. Now, for a closed subgroup $H$ of $G$, we define the following seminorm on $L^1(G)$:
\begin{align}\label{norm_A_c(G,H)}
	\|f\|_{\tilde{A}_c(G,H)}=\sup_{\pi\in\mathcal{R}_c(G,H)}\|\pi(f)\|.
\end{align}
This will not be a norm in general. For instance, when $H=\{e\}$ is the trivial subgroup, $\mathcal{R}_c(G,H)$ contains only multiples of the trivial representation, and
\begin{align*}
	\|f\|_{\tilde{A}_c(G,\{e\})}=\left|\int_G f(s)\,ds\right|.
\end{align*}
We define $\tilde{A}_c(G,H)$ as the completion of $L^1(G)/N_H$ for the quotient norm induced by \eqref{norm_A_c(G,H)}, where
\begin{align*}
	N_H=\{f\in L^1(G)\ \mid\ 	\|f\|_{\tilde{A}_c(G,H)}=0\}.
\end{align*}
We also define $B_c(G,H)$ as the dual Banach space $\tilde{A}_c(G,H)^*$. Observe that the identity map $L^1(G)\to L^1(G)$ extends to a contraction $S_H:\tilde{A}_c(G)\to \tilde{A}_c(G,H)$ with dense range. Hence the adjoint map $S_H^*:B_c(G,H)\to B_c(G)$ is a contractive inclusion. Therefore we may view $B_c(G,H)$ as a subspace of $B_c(G)$. Moreover, $B_c(G,H)$ contains all matrix coefficients of representations in $\mathcal{R}_c(G,H)$; see Lemma \ref{Lem_Bc_dense}.

Since $B_c(G)$ is a translation-invariant subspace of $\operatorname{WAP}(G)$ that contains the constant functions and is closed under complex conjugation, it carries a unique $G$-invariant mean; see \cite[\S 2.2]{Ver} for details. Therefore, the same holds for $B_c(G,H)$.

\subsection{Isometric representations on uniformly convex Banach spaces}
If $E$ is a Banach space, let $S_E$ be its unit sphere, $B_E$ its unit ball, and $\mathbf{O}(E)$ the group of all linear invertible isometries.
\begin{defi}\label{def:superref}
Let $(E,\Vert\cdot\Vert)$ be a Banach space and $B_E$ its unit ball. Let \begin{equation}
    \label{eq:modulusconvexity} d_{\Vert \cdot\Vert}(\varepsilon)=\inf \left\lbrace 1-\frac{\Vert \xi+\eta\Vert}{2} \mid \xi,\eta\in B_E, \ \Vert \xi-\eta\Vert \geq \varepsilon\right\rbrace
\end{equation}and \begin{equation}
    \label{eq:modulussmoothness} r_{\Vert \cdot\Vert}(\tau)=\sup \left\lbrace \frac{\Vert\xi+\eta\Vert+\Vert \xi-\eta\Vert}{2}-1 \mid \xi\in B_E, \ \Vert \eta\Vert \leq \tau\right\rbrace.
\end{equation}
If there is no ambiguity, we may denote $d(\varepsilon)$ and $r(\tau)$. The functions $d$ and $r$ are called the \textit{modulus of convexity} and \textit{modulus of smoothness} respectively.

\begin{enumerate}
    \item The space $E$ is said to be uniformly convex (or uc) if for any $0<\varepsilon<2$, $d(\varepsilon)>0$.
    \item The space $E$ is said to be uniformly smooth (or us) if $\underset{\tau\to 0}{\lim} r(\tau)/\tau=0$.
    \item The space $E$ is said to be ucus if it is both uc and us.
    \item The space $E$ is said to be super-reflexive if it admits a compatible ucus norm.
\end{enumerate}
\end{defi}

The following result was proved in \cite[Proposition 2.3]{BFGM}.
\begin{prop}[Bader--Furman--Gelander--Monod]\label{prop:ubtoisom}
Let $E$ be a super-reflexive Banach space and let $\rho$ be a uniformly bounded representation on $E$. Then there exists a compatible ucus norm $\Vert\cdot\Vert$ on $E$ such that $\rho:G\to \mathbf{O}(E,\Vert \cdot\Vert)$ is an isometric representation.
\end{prop}
\smallskip
\begin{rmk}\label{Rmk_[H]_c}
In particular, denote $[\mathcal{H}]$ the class of all Banach spaces with a compatible Hilbert norm. Also, denote $[\mathcal{H}]_c$ the subclass of all spaces with a norm ratio $\leq c$, that is to say with $$\max\left(\frac{\Vert x\Vert_H}{\Vert x\Vert_B},\frac{\Vert x\Vert_B}{\Vert x\Vert_H}\right)\leq c$$where $\Vert\cdot\Vert_B,\Vert\cdot\Vert_H$ are the original and the Hilbert norm respectively. 

Then the above construction shows that if $\rho$ is a uniformly bounded representation on a Hilbert space $\mathcal{H}$, it can be viewed as an isometric representation on a space in $[\mathcal{H}]$, and vice-versa - and we can check that if $|\rho|\leq c$, we can take this space in $[\mathcal{H}]_c$.  In particular property (T)${}_c$ is equivalent to property $(T_{[\mathcal{H}]_c})$ in the sense of \cite{BFGM}, and property (T${}_{ub}$) is equivalent to property $(T_{\mathcal{[H]}})$.
\end{rmk}

If $E$ is a Banach space, let $\langle \cdot, \cdot \rangle:E\times E^*\to \C$ denote the duality bracket of $E$.
\begin{defi}\label{def:duality}
Let $E$ be a us Banach space and $S_E$ its unit sphere. Then for $\xi\in S_E$, there exists a unique element $\xi^*\in S_{E^*}$ such that $\langle \xi,\xi^*\rangle =1$. The map $\xi\mapsto \xi^*$ is called the duality mapping.
\end{defi}
The duality mapping can be extended to all of $E$ by multiplication with a real number; however, it does not possess a linearity property in general.

The following result can be found in \cite[Proposition A.5]{BenLin}.
\begin{lem}\label{lem:duality}
Let $E$ be a uniformly smooth Banach space. Let $0<\varepsilon<2$, then for any $\xi,\eta\in S_E$ with $\Vert\xi-\eta\Vert \leq \varepsilon$, we have $$\Vert \xi^*-\eta^*\Vert \leq 2 \frac{r(2\varepsilon)}{\varepsilon}.$$In particular, the duality mapping is uniformly continuous.
\end{lem}

\begin{defi}\label{def:contragredient}
    Let $\rho$ be a representation of $G$ on a Banach space $E$. We define the contragredient representation $\rho^*$ by \[\forall g\in G,\xi\in E,\phi\in E^*,\quad \langle \xi,\rho^*(g)\phi\rangle= \langle \rho(g^{-1})\xi,\phi\rangle.\]
\end{defi}

\begin{rmk}\label{rmk:contrag}
    If $\rho$ is an isometric representation on a Banach space $E$, so is $\rho^*$. Furthermore, if $E$ is uniformly smooth, then $(\rho(g)\xi)^*=\rho^*(g)\xi^*$ for any $\xi\in S_E$. Indeed, $\Vert \rho^*(g)\xi^*\Vert=1$ and $\langle \rho(g)\xi,\rho^*(g)\xi^*\rangle=\langle \xi,\xi^*\rangle=1$ which proves the claim by uniqueness.
\end{rmk}

\begin{rmk}\label{rk:renorm} A Hilbert space $(\mathcal{H},\Vert\cdot\Vert)$ is ucus with \[d(\varepsilon)=1-\sqrt{1-\frac{\varepsilon^2}{2}}\quad \textrm{  and  }\quad r(\tau)=\sqrt{1+\tau^2}-1.\]
    We point out that in Proposition \ref{prop:ubtoisom}, if $\rho$ is a representation on the Hilbert space $(\mathcal{H},\Vert\cdot\Vert)$ with $\vert\rho\vert\leq c$, one can take $\Vert \cdot\Vert_\rho$ to be the (pre)dual norm of $\Vert\phi\Vert_{\rho^*}=\underset{g}{\sup}\Vert \phi\circ \rho(g)\Vert$. Then $\Vert\cdot\Vert_\rho$ is equivalent to $\Vert\cdot \Vert$ with norm ratio at most $c$, and $(\mathcal{H},\Vert\cdot\Vert_\rho)$ is uniformly smooth, with $$r_{\Vert\cdot\Vert_\rho}(\tau)\leq \frac{c^2\tau^2}{2}.$$When dealing with a specific group, this knowledge can help give precise Kazhdan constants in property (T)${}_c$ (see, for instance, \cite{Mim}). This specific norm will also be useful in Sections \ref{sec:ab} and \ref{sec:nil}.
\end{rmk}

\subsection{Complex measures}\label{sec:measures} 
Let $(X,\Sigma)$ be a measurable space. A complex measure is a map $\mu:\Sigma\to \C$ such that for any sequence $(A_n)_{n\in \N}\subset \Sigma$ of disjoint sets, $\mu(\bigcup A_n)=\sum_{n\geq 0} \mu(A_n)$ (see \cite[Ch. 6]{Rud}). Notice that this implies the absolute convergence of $\sum \mu(A_n)$. We define the total variation measure $\vert \mu\vert$ by $$\vert\mu\vert(A)=\sup \left\lbrace \sum \vert \mu(A_k)\vert \mid (A_k) \textrm{ partition of }A\right\rbrace.$$Then $\vert \mu\vert$ is a finite positive measure and is the smallest positive measure $\nu$ such that for all $A\in \Sigma$, $\vert \mu(A)\vert\leq \nu(A)$. It is easy to show that $\vert \mu_1+\mu_2\vert\leq \vert \mu_1\vert+\vert \mu_2\vert$ so $$\left\vert \vert\mu_1\vert-\vert\mu_2\vert \right\vert\leq \vert \mu_1-\mu_2\vert.$$ Furthermore, $\mu$ has a polar decomposition: there exists a measurable map $\theta:X\to \R$ such that $$d\mu=e^{i\theta}d\vert \mu\vert.$$
We define a norm on the space of all complex measures by $$\Vert \mu\Vert_{TV}=\vert \mu\vert (X).$$

We can decompose $\mu=\Re \mu+i\Im \mu$ with $\Re \mu,\Im \mu$ real signed measures. We can then consider the Hahn--Jordan decompositions $\Re \mu=\mu_{r}^+-\mu_{r}^-$ and $\Im \mu=\mu_{i}^+-\mu_{i}^-$. Later, we will often denote $\mu^+=\mu_{r}^+$. The maps $\mu\mapsto \Re \mu,\mu\mapsto \Im \mu$ are $\R$-linear and of norm $1$. Furthermore, for a signed measure $\sigma$, $\sigma_{+}=\frac{\vert \sigma\vert+\sigma}{2}$ so that $\sigma\mapsto \sigma_+$ is $1$-Lipschitz. Thus, the map $\mu\mapsto \mu_r^+$ is $1$-Lipschitz, meaning that for any $\mu,\nu$, \begin{equation}
    \label{eq:lippospart} \vert \mu_r^+-\nu_r^+\vert \leq \vert \mu-\nu\vert.
\end{equation}

Let $\Tilde{\mu}=\mu_{r}^++\mu_{r}^-+\mu_{i}^++\mu_{i}^-=\vert\Re \mu\vert+\vert\Im \mu\vert=(\vert \cos \theta\vert+\vert \sin \theta\vert)\vert\mu\vert$. Then for any $A\in \Sigma$, \begin{equation}\label{eq:totalvarineq}
    \frac{1}{\sqrt{2}}\Tilde{\mu}(A) \leq \vert \mu\vert(A)\leq \Tilde{\mu}(A).
\end{equation}
For any $A\in \Sigma$, we have $\vert\Re \mu(A)-\Re \nu(A)\vert \leq \vert \mu(A)-\nu(A)\vert$, same with $\Im \mu,\Im \nu$. Furthermore, let $\sigma=\Re \mu-\Re \nu$ and $X=P\cup N$ the Hahn decomposition of the signed measure $\sigma$. Then $\vert\sigma\vert(X)=\vert \sigma(P)\vert+\vert\sigma(N)\vert$. Thus we have $\Vert \Re \mu-\Re \nu\Vert_{TV}\leq 2\sup \vert \Re \mu(A)-\Re \nu(A)\vert$ so that \begin{equation}\label{eq:supbound}
    \Vert \mu_+-\nu_+\Vert_{TV} \leq \Vert \Re \mu-\Re \nu\Vert_{TV}\leq 2\sup \vert \mu(A)-\nu(A)\vert. 
\end{equation}

Assume now that $X$ is locally compact and $\Sigma$ is its Borel $\sigma$-algebra. Let $M(X)$ be the space of complex measures $\mu$ such that $\vert\mu\vert$ is regular. Then $(M(X),\Vert\cdot\Vert_{TV})$ is naturally isometric to $C_0(X)^*$ through integration against $\mu$.

Let also $\mathcal{L}^\infty(X)$ be the space of all bounded measurable functions on $X$ with the supremum norm. Then $\mathcal{L}^\infty(X)^*$ is identified with the space of \textit{finitely additive} finite complex measures. A mean is a functional $m\in \mathcal{L}^\infty(X)^*$ that is positive and such that $m(1)=1$.

\subsection{Harmonic analysis on locally compact abelian groups}
Let $A$ be a locally compact abelian group and $\hat{A}$ its Pontryagin dual, that is to say, the group of characters on $A$. Let $\mathcal{B}(\hat{A})$ denote the Borel $\sigma$-algebra on $\hat{A}$. Let $\pi:A\to \mathbf{U}(\mathcal{H}_\pi)$ be a unitary representation of $A$. The SNAG theorem \cite[Section D.3]{BHV} yields a projection-valued measure $E:\mathcal{B}(\hat{A})\to \mathbf{B}(\mathcal{H}_\pi)$. Given a bounded Borel function $f$, we construct an operator \[\pi(f)=\int_{\hat{A}} f(\chi)\ dE(\chi)\in \mathbf{B}(\mathcal{H}_\pi).\]
In particular, with $f=\mathrm{ev}_a$ the evaluation map $\chi\mapsto \chi(a)$, we get that for any $a\in A$, \[\pi(a)=\int_{\hat{A}}\chi(a) dE(\chi).\]

The notation $\pi(f)$ is justified by the following point of view: one can define a representation $\pi:L^1(A)\to \mathbf{B}(\mathcal{H}_\pi)$ by $f\mapsto \int_A f(a)\pi(a)\ da$. This representation extends to a representation of the reduced $C^*$-algebra $C^*_r(A)$. But since $A$ is abelian, the Fourier transform implements an isomorphism $C^*_r(A)\simeq C_0(\hat{A})$, and the representation of $C_0(\hat{A})$ induced by $\pi$ by composition with this isomorphism is exactly the $*$-homomorphism $\pi$ given by the projection-valued measure $E$.\medskip

The projection-valued measure allows us to relate objects defined on $A$ to measures on $\hat{A}$. This fact is at the heart of the proof of the following theorem of Cornulier and Tessera \cite{CorTes}, which gives a complete characterisation of property (T) for semidirect products of $A$. If a topological group $G$ acts continuously on $A$ by group automorphisms, then $G$ also acts on the Pontryagin dual $\hat{A}$ by precomposition: for any $\chi\in \hat{A}$, $g\in G$ and $a\in A$, $(g\chi)(a)=\chi(g^{-1}a)$. Then $G$ acts on Borel measures on $\hat{A}$ as well. If $\mu$ is a Borel measure on $\hat{A}$, the measure $g_*\mu$ is the pushforward of $\mu$ by the homeomorphism $\chi\mapsto g\chi$ of $\hat{A}$.
\begin{thm}[Cornulier--Tessera]\label{thm:cortes} Let $G=H\ltimes A$ be a locally compact group with $A$ abelian. Let $\hat{A}$ be the Pontryagin dual of $A$ and $1\in \hat{A}$ the trivial character of $A$. The following are equivalent:
\begin{enumerate}
    \item[($\neg T$)] The pair $(G,A)$ does not have relative property (T)
    \item[$(P)$] There exists a net of Borel probability measures $(\mu_i)$ on $\hat{A}$ such that
    \begin{enumerate}
        \item[$(P1)$] $\mu_i$ converges in weak* topology to the Dirac mass $\delta_1$,
        \item[$(P2)$] $\mu_i(\{1\})=0$,
        \item[$(P3)$] $\Vert h_*\mu_i-\mu_i\Vert_{TV}\to 0$ uniformly on compact subsets of $H$.
    \end{enumerate}
\end{enumerate}
Furthermore, if $A$ is $\sigma$-compact, $(P1)$ is equivalent to \[(P1')\quad \int_{\hat{A}}\chi(a)\ d\mu_i(\chi)\to 1 \textrm{ uniformly on compact subsets of }A.\]
\end{thm}

When $A$ is an amenable group (in particular, when $A$ is abelian), it is well known that any uniformly bounded representation of $A$ is unitarisable, i.e. if $\pi$ is a uniformly bounded representation on $\mathcal{H}_\pi$, then there exists an invertible operator $T\in \mathbf{B}(\mathcal{H}_\pi)$ such that $\pi':a\mapsto T\circ \pi(a)\circ T^{-1}$ is unitary. Furthermore, $\Vert T\Vert \Vert T^{-1}\Vert\leq |\pi|^2$ \cite{Day, Dix}.

\begin{rmk}
    If any uniformly bounded representation of a group $G$ is unitarisable, we say that $G$ is unitarisable. The above result, proved independently by Day \cite{Day}, Dixmier \cite{Dix}, and Nakamura--Takeda \cite{NakTak}, says that all amenable groups are unitarisable. The converse is a long-standing open question of Dixmier: is every unitarisable group amenable? See \cite{Pis} for an introduction to this question.
\end{rmk}

\section{Characterisation of relative property \texorpdfstring{(T)${}_c$}{(T)c}}\label{sec:charct}

In this section, we prove Theorem \ref{Thm_char_(T)c_intro} and Proposition \ref{Prop_G/H_(G,H)}.

\subsection{Proof of Theorem \ref{Thm_char_(T)c_intro}}
The following characterisation of the unique invariant mean on $B_c(G)$ will be useful.

\begin{lem}\label{Lem_char_p_pi}
	Let $G$ be a locally compact group, $c\geq 1$, and let $m$ denote the unique invariant mean on $B_c(G)$. For every $\pi\in\mathcal{R}_c(G)$, the projection onto $\mathcal{H}_\pi^G$ along $\left(\mathcal{H}_{\pi^*}^G\right)^\perp$ is the unique operator $p_\pi\in\mathbf{B}(\mathcal{H}_\pi)$ satisfying
	\begin{align*}
		\langle p_\pi\xi,\eta\rangle=m(s\mapsto\langle\pi(s)\xi,\eta\rangle)
	\end{align*}
	for all $\xi,\eta\in\mathcal{H}_\pi$. As a consequence, $\|p_\pi\|\leq c$.
\end{lem}
\begin{proof}
	First observe that the formula
	\begin{align*}
		\langle p_\pi\xi,\eta\rangle=m(s\mapsto\langle\pi(s)\xi,\eta\rangle)
	\end{align*}
	gives a well defined operator $p_\pi\in\mathbf{B}(\mathcal{H}_\pi)$ because
	\begin{align*}
		|\langle p_\pi\xi,\eta\rangle|\leq\sup_{t\in G} |\langle\pi(t)\xi,\eta\rangle|\leq c\|\xi\| \|\eta\|,
	\end{align*}
	which also shows that $\|p_\pi\|\leq c$. Let us prove now that $p_\pi$ is the projection onto $\mathcal{H}_\pi^G$ along $\left(\mathcal{H}_{\pi^*}^G\right)^\perp$. First, if $\xi\in \mathcal{H}_\pi^G$, then, for all $\eta\in\mathcal{H}_{\pi}$,
	\begin{align*}
		\langle p_\pi\xi,\eta\rangle&=m(s\mapsto\langle\xi,\eta\rangle)\\
		&=\langle\xi,\eta\rangle,
	\end{align*}
	which shows that $p_\pi\xi=\xi$. On the other hand, for all $s\in G$, $\xi,\eta\in \mathcal{H}_\pi$,
	\begin{align*}
		\langle\pi(s)p_\pi\xi,\eta\rangle&=m\left(t\mapsto \langle\pi(t)\xi,\pi(s)^*\eta\rangle\right)\\
		&=m\left(t\mapsto \langle\pi(st)\xi,\eta\rangle\right)\\
		&=m\left(t\mapsto \langle\pi(t)\xi,\eta\rangle\right)\\
		&=\langle p_\pi\xi,\eta\rangle,
	\end{align*}
	which shows that $p_\pi\xi$ belongs to $\mathcal{H}_\pi^G$. Hence $p_\pi$ is a projection onto $\mathcal{H}_\pi^G$. Now let $\xi\in\left(\mathcal{H}_{\pi^*}^G\right)^\perp$. For every $\eta\in \mathcal{H}_{\pi}$,
	\begin{align*}
		\langle p_\pi\xi,\eta\rangle&=m\left(t\mapsto \langle\xi,\pi(t)^*\eta\rangle\right)\\
		&=m\left(t\mapsto \langle\xi,\pi^*(t^{-1})\eta\rangle\right)\\
		&=m\left(t\mapsto \langle\xi,\pi^*(t)\eta\rangle\right)\\
		&=\langle \xi,p_{\pi^*}\eta\rangle,
	\end{align*}
	where $p_{\pi^*}\in\mathbf{B}(\mathcal{H}_{\pi^*})$ is the projection onto $\mathcal{H}_{\pi^*}^G$ along $\left(\mathcal{H}_{\pi}^G\right)^\perp$. Here we used the fact that $m$ is inversion invariant; see \cite[Corollary 2.6]{Ver}. By the previous argument, $p_{\pi^*}\eta$ belongs to $\mathcal{H}_{\pi^*}^G$. Hence
	\begin{align*}
		\langle p_\pi\xi,\eta\rangle=0,
	\end{align*}
	and therefore $p_\pi\xi=0$. This shows that $p_\pi$ is the projection onto $\mathcal{H}_\pi^G$ along $\left(\mathcal{H}_{\pi^*}^G\right)^\perp$. 
\end{proof}

We will also need the following density result.

\begin{lem}\label{Lem_Bc_dense}
	Let $G$ be a locally compact group, $H$ a closed subgroup of $G$, and $c\geq 1$. Let $\tilde{B}_c(G,H)$ denote the subspace of $B_c(G)$ given by matrix coefficients of representations in $\mathcal{R}_c(G,H)$. Then $\tilde{B}_c(G,H)$ is a $\sigma(B_c(G,H),\tilde{A}_c(G,H))$-dense subspace of $B_c(G,H)$.
\end{lem}
\begin{proof}
	Let $\varphi\in \tilde{B}_c(G,H)$ be given by $\varphi(s)=\langle\pi(s)\xi,\eta\rangle$ with $\pi\in\mathcal{R}_c(G,H)$ and $\xi,\eta\in\mathcal{H}_\pi$. For every $f\in L^1(G)$,
	\begin{align*}
		|\langle\varphi, f\rangle| &= |\langle\pi(f)\xi,\eta\rangle|\\
		&\leq \|f\|_{\tilde{A}_c(G,H)}\|\xi\| \|\eta\|.
	\end{align*}
	This shows that $\varphi$ defines an element of $\tilde{A}_c(G,H)^*=B_c(G,H)$. Therefore $\tilde{B}_c(G,H)$ is a subspace of $B_c(G,H)$. Now let $f\in \tilde{A}_c(G,H)$ such that, for every $\varphi\in\tilde{B}_c(G,H)$,
	\begin{align*}
		\langle\varphi,f\rangle=0.
	\end{align*}
	Let $\pi\in\mathcal{R}_c(G,H)$. For every $\xi,\eta\in\mathcal{H}_\pi$, define $\varphi_{\xi,\eta}\in\tilde{B}_c(G,H)$ by
	\begin{align*}
		\varphi_{\xi,\eta}(s)=\langle\pi(s)\xi,\eta\rangle.
	\end{align*}
	Then
	\begin{align*}
		0&=\langle\varphi_{\xi,\eta},f\rangle\\
		&=\langle\pi(f)\xi,\eta\rangle. 
	\end{align*}
	This shows that $\pi(f)=0$. Since this holds for every $\pi\in\mathcal{R}_c(G,H)$, we conclude that $f=0$. By the Hahn--Banach theorem, $\tilde{B}_c(G,H)$ is weak*-dense in $B_c(G,H)$.
\end{proof}

\begin{rmk}
    The definition of $\tilde{B}_c(G,H)$ in Lemma \ref{Lem_Bc_dense} might seem like a more natural generalisation of $B_c(G)$ to the relative setting than $B_c(G,H)$. One can try to adapt the proof of \cite[Proposition 2.10]{Ver} in order to show that $\tilde{A}_c(G,H)^*=\tilde{B}_c(G,H)$; however, this approach does not work because the class $\mathcal{R}_c(G,H)$ is not stable under ultraproducts. We avoid all these issues by simply defining $B_c(G,H)$ as the dual space of $\tilde{A}_c(G,H)$.
\end{rmk}

In order to characterise relative property (T)${}_c$ in terms of projections, we will translate the spectral gap condition \eqref{spec_gap} to the language of \cite{DruNow}. This will allow us to prove the equivalence (i)$\iff$(ii) in Theorem \ref{Thm_char_(T)c_intro} by applying Theorem \ref{Thm_DruNow} below. Let $G$ be a locally compact group and $c\geq 1$. For every $\pi\in\mathcal{R}_c(G)$, define a new norm on $\mathcal{H}_\pi$ by
\begin{align}\label{norm_E_pi}
	\|\xi\|_{E_\pi}=\sup_{s\in G}\|\pi(s)\xi\|_{\mathcal{H}_\pi}.
\end{align}
The Banach space $E_\pi$ thus obtained is uniformly convex, $c$-isomorphic to $\mathcal{H}_\pi$, and the representation $\pi$ is isometric on $E_\pi$; see \cite[Proposition 2.3]{BFGM} for details. Observe that $E_\pi$ belongs to the class $[\mathcal{H}]_c$; see Remark \ref{Rmk_[H]_c}. In this case, the decomposition \eqref{dec_G-inv} can be written as
\begin{align*}
	E_\pi=E_\pi^G\oplus \left(E_{\pi^*}^G\right)^\perp,
\end{align*}
where $\left(E_{\pi^*}^G\right)^\perp$ denotes the subspace of $(E_\pi)^{**}=E_\pi$ given by all functionals that vanish on the subspace of $G$-invariant vectors $E_{\pi^*}^G\subseteq (E_\pi)^{*}$; see \cite[\S 2.3]{DruNow} for details. Moreover, since the spaces $E_\pi$ and $\mathcal{H}_\pi$ are $c$-isomorphic, the spectral gap condition \eqref{spec_gap} for $\left(\mathcal{H}_{\pi^*}^G\right)^\perp$ and a pair $(Q,\varepsilon)$ is equivalent to the same condition for $\left(E_{\pi^*}^G\right)^\perp$:
\begin{align}\label{spec_gap_isom}
	\max_{s\in Q}\|\pi(s)\xi-\xi\|_{E_\pi}\geq \varepsilon'\|\xi\|_{E_\pi},
\end{align}
for the same compact set $Q\subseteq G$ and some constant $\varepsilon'>0$. More precisely, if $\pi$ satisfies \eqref{spec_gap} for $(Q,\varepsilon)$, then it satisfies \eqref{spec_gap_isom} for $(Q,\varepsilon/c)$. Conversely, if it satisfies \eqref{spec_gap_isom} for $(Q,\varepsilon')$, then it satisfies \eqref{spec_gap} for $(Q,\varepsilon'/c)$.

Let $\mathcal{F}$ be a class of isometric representations of a locally compact group $G$ on a uniformly convex family of Banach spaces $\mathcal{E}$. This means that, for every $\varepsilon\in(0,2)$,
\begin{align*}
	\inf_{E\in\mathcal{E}} d_{E}(\varepsilon)>0,
\end{align*}
where $d_{E}$ denotes the modulus of convexity of $E$, as defined in \eqref{eq:modulusconvexity}. We define $C_{\mathcal{F}}(G)$ as the completion of $L^1(G)$ (after modding out by the null space) for the norm
\begin{align*}
	\|f\|_{C_{\mathcal{F}}(G)}=\sup_{\pi\in\mathcal{F}}\|\pi(f)\|.
\end{align*}
Observe that, if $H$ is a closed subgroup of $G$, and $c\geq 1$, then
\begin{align*}
	\mathcal{F}=\left\{(E_\pi,\pi)\ \mid\ \pi\in\mathcal{R}_c(G,H)\right\}
\end{align*}
 is a class of isometric representations on a uniformly convex family of Banach spaces, where $E_\pi$ is constructed from $\mathcal{H}_\pi$ as in \eqref{norm_E_pi}. Moreover, in this case, $C_{\mathcal{F}}(G)$ is isomorphic (as a Banach algebra) to $\tilde{A}_c(G,H)$.
 
 The following result was proved in \cite[Theorem 4.6]{DruNow}.
 
 \begin{thm}[Dru\c{t}u--Nowak]\label{Thm_DruNow}
 	Let $G$ be a locally compact group, and let $\mathcal{F}$ be a class of isometric representations of $G$ on a uniformly convex family of Banach spaces $\mathcal{E}$. Then the following are equivalent:
 	\begin{itemize}
 		\item[(i)] There is a compact subset $Q\subseteq G$ and $\varepsilon>0$ such that, for every $\pi\in\mathcal{F}$ and every $\xi\in\left(E_{\pi^*}^G\right)^\perp$,
 		\begin{align*}
 			\sup_{s\in Q}\|\pi(s)\xi-\xi\|\geq\varepsilon\|\xi\|.
 		\end{align*}
 		\item[(ii)] There exists an idempotent $P\in C_{\mathcal{F}}(G)$ such that
 		\begin{itemize}
 			\item for every $\pi\in\mathcal{F}$, $\pi(P)$ is the projection onto $E_\pi^G$ along $\left(E_{\pi^*}^G\right)^\perp$,
 			\item there exists a sequence of continuous, compactly supported probability measures $(\rho_n)$ on $G$ such that
 			\begin{align*}
 				\lim_{n\to\infty}\|\rho_n-P\|_{C_{\mathcal{F}}(G)}=0.
 			\end{align*}
 		\end{itemize}
 	\end{itemize}
 \end{thm}

With this, we can prove Theorem \ref{Thm_char_(T)c_intro}.

\begin{proof}[Proof of Theorem \ref{Thm_char_(T)c_intro}]
	(i)$\implies$(ii): Assume that the pair $(G,H)$ satisfies property (T)${}_c$, and let $(Q,\varepsilon)$ be a Kazhdan pair for $(\mathcal{R}_c(G),H)$. For every $\pi\in\mathcal{R}_c(G,H)$, we have
	\begin{align*}
		\mathcal{H}_\pi=\mathcal{H}_\pi^G\oplus \left(\mathcal{H}_{\pi^*}^G\right)^\perp,
	\end{align*}
	and $\mathcal{H}_\pi^H=\mathcal{H}_\pi^G$. By restricting $\pi$ to $\left(\mathcal{H}_{\pi^*}^G\right)^\perp=\left(\mathcal{H}_{\pi^*}^H\right)^\perp$, we obtain a representation without non-trivial  $H$-invariant vectors, and therefore, from the definition of Kazhdan pair, we get
	\begin{align*}
		\sup_{s\in Q}\|\pi(s)\xi-\xi\|_{\mathcal{H}_\pi}\geq\varepsilon\|\xi\|_{\mathcal{H}_\pi},
	\end{align*}
	for every $\xi\in\left(\mathcal{H}_{\pi^*}^G\right)^\perp$. As discussed above, this implies that
	\begin{align*}
		\max_{s\in Q}\|\pi(s)\xi-\xi\|_{E_\pi}\geq \frac{\varepsilon}{c}\|\xi\|_{E_\pi},
	\end{align*}
	for every $\xi\in\left(E_{\pi^*}^G\right)^\perp$. By Theorem \ref{Thm_DruNow}, this is equivalent to the existence of a Kazhdan projection $P\in C_{\mathcal{F}}(G)$, where
	\begin{align*}
		\mathcal{F}=\left\{(E_\pi,\pi)\ \mid\ \pi\in\mathcal{R}_c(G,H)\right\}.
	\end{align*}
	By the isomorphism between $C_{\mathcal{F}}(G)$ and $\tilde{A}_c(G,H)$, we obtain an idempotent $P\in\tilde{A}_c(G,H)$ with the desired properties.\\
	(ii)$\implies$(i): Since Theorem \ref{Thm_DruNow} gives an equivalence, the argument above can be repeated in the opposite direction. We leave the verification of this fact to the reader. Instead, we present here a direct proof of the implication (ii)$\implies$(i) that only uses classical arguments. Let $P\in\tilde{A}_c(G,H)$ be the Kazhdan projection, and let $\rho$ be a continuous, compactly supported probability measure on $G$ such that
	\begin{align*}
		\|\rho-P\|_{\tilde{A}_c(G,H)}<\frac{1}{3}.
	\end{align*}
	If we assume by contradiction that the pair $(G,H)$ does not satisfy Property (T)${}_c$, then there is $\pi\in\mathcal{R}_c(G)$ with $\mathcal{H}_\pi^H=\{0\}$, and a unit vector $\xi\in \mathcal{H}_\pi$ such that
	\begin{align*}
		\max_{s\in Q}\|\pi(s)\xi-\xi\|<\frac{1}{3},
	\end{align*}
	where $Q=\operatorname{supp}(\rho)$. Notice that, in this case, $\pi$ belongs to $\mathcal{R}_c(G,H)$ because $\mathcal{H}_\pi^G\subseteq\mathcal{H}_\pi^H=\{0\}$. Since $\pi(P)\xi=0$,
	\begin{align*}
		1&=\|\xi\|\\
		&=\|\xi-\pi(P)\xi\|\\
		&\leq\|\xi-\pi(\rho)\xi\|+\|\pi(\rho)\xi-\pi(P)\xi\|\\
		&\leq \left\|\int_Q\rho(s)(\xi-\pi(s)\xi)\, ds\right\| + \|\pi(\rho)-\pi(P)\|_{\mathbf{B}(\mathcal{H}_\pi)}\\
		&\leq \int_Q\rho(s)\|\xi-\pi(s)\xi\|\, ds + \|\rho-P\|_{\tilde{A}_c(G,H)}\\
		&\leq \frac{2}{3},
	\end{align*}
	which gives a contradiction.\\
	(ii)$\implies$(iii): Let $P\in\tilde{A}_c(G,H)$ be the Kazhdan projection. We will show that $P$ defines an invariant mean on $B_c(G,H)$. Let $(\rho_n)$ be the sequence of probability measures that approximates $P$. Then
	\begin{align*}
		\langle 1,P\rangle=\lim_n\langle 1,\rho_n\rangle= 1.
	\end{align*}
	Moreover, for every positive function $\varphi\in B_c(G,H)$,
	\begin{align*}
		\langle\varphi,P\rangle = \lim_n\langle \varphi,\rho_n\rangle \geq 0.
	\end{align*}
	This shows that $P$ is a mean. In order to show that it is invariant, we will first restrict it to the weak*-dense subspace $\tilde{B}_c(G,H)\subseteq B_c(G,H)$ defined in Lemma \ref{Lem_Bc_dense}. Let $\varphi\in \tilde{B}_c(G,H)$ be given by $\varphi(s)=\langle\pi(s)\xi,\eta\rangle$ with $\pi\in\mathcal{R}_c(G,H)$ and $\xi,\eta\in\mathcal{H}_\pi$. Then
	\begin{align*}
		\langle\varphi,P\rangle_{B_c(G,H),\tilde{A}_c(G,H)}=\langle\pi(P)\xi,\eta\rangle.
	\end{align*}
	Moreover, since $\pi(P)$ is the projection onto $\mathcal{H}_\pi^G$ along $\left(\mathcal{H}_{\pi^*}^G\right)^\perp$, by Lemma \ref{Lem_char_p_pi},
	\begin{align*}
		\langle\pi(P)\xi,\eta\rangle&=m(s\mapsto\langle\pi(s)\xi,\eta\rangle)\\
		&=\langle m,\varphi\rangle_{B_c(G,H)^*,B_c(G,H)},
	\end{align*}
	where $m$ is the unique invariant mean on $B_c(G,H)$. Hence, $P$ defines an invariant mean on $\tilde{B}_c(G,H)$. Now let $\varphi\in B_c(G,H)$ and $s\in G$. We need to prove that
	\begin{align*}
		\langle s\cdot\varphi,P\rangle=\langle\varphi,P\rangle,
	\end{align*}
	where $s\cdot\varphi(t)=\varphi(s^{-1}t)$. Observe that, for every $f\in L^1(G)$ and $\pi\in\mathcal{R}_c(G,H)$,
	\begin{align*}
		\|\pi(\delta_{s^{-1}}\ast f)\|=\|\pi(s^{-1})\pi(f)\|\leq c\|\pi(f)\|,
	\end{align*}
	which shows that the convolution with $\delta_{s^{-1}}$ extends to a bounded operator on $\tilde{A}_c(G,H)$ of norm at most $c$. Now take a net $(\varphi_i)$ in $\tilde{B}_c(G,H)$ converging to $\varphi$ in the weak*-topology. We obtain
	\begin{align*}
		\langle s\cdot\varphi,P\rangle &= \langle \varphi,\delta_{s^{-1}}\ast P\rangle\\
		&= \lim_i \langle \varphi_i,\delta_{s^{-1}}\ast P\rangle\\
		&= \lim_i \langle s\cdot\varphi_i,P\rangle\\
		&= \lim_i \langle \varphi_i,P\rangle\\
		&=\langle\varphi,P\rangle.
	\end{align*}
	This shows that $P$ is an invariant mean on $B_c(G,H)$. By the uniqueness of this mean, we conclude that $m=P$. In particular, $m$ is weak*-continuous.\\
	(iii)$\implies$(ii): We assume that the unique invariant mean $m\in B_c(G,H)^*$ belongs to $\tilde{A}_c(G,H)$. By Lemma \ref{Lem_char_p_pi}, for every $\pi\in\mathcal{R}_c(G,H)$, $\pi(m)$ is the projection onto $\mathcal{H}_\pi^G$ along $\left(\mathcal{H}_{\pi^*}^G\right)^\perp$. On the other hand, the same argument as in \cite[Lemma 4.2]{Ver} shows that $m$ is the limit of a sequence of continuous, compactly supported probability measures on $G$.
\end{proof}

\subsection{Proof of Proposition \ref{Prop_G/H_(G,H)}}
We can now prove Proposition \ref{Prop_G/H_(G,H)} using Kazhdan projections. We will actually prove the following, more general result, which will be needed in the proof of Theorem \ref{thm:mainthmnil}.

\begin{lem}\label{Lem_(G/H,K/H)}
	Let $G$ be a locally compact group, $H, K$ two closed, normal subgroups of $G$ such that $H\leq K$, and $c\geq 1$. Then $(G,K)$ has relative property (T)${}_c$ if and only if both $(G/H,K/H)$ and $(G,H)$ have relative property (T)${}_c$.
\end{lem}

In order to prove Lemma \ref{Lem_(G/H,K/H)}, we will need some preliminary results. The following lemma generalises \cite[Lemma 6.2]{Ver} and \cite[Corollary 6.3]{Ver}.

\begin{lem}\label{Lem_T_H}
	Let $G$ be a locally compact group, $H, K$ two closed, normal subgroups of $G$ such that $H\leq K$, and $c\geq 1$. Then the map $T_H:L^1(G)\to L^1(G/H)$, given by
	\begin{align}\label{def_T_H}
		T_Hf(\dot{s})=\int_{H}f(sx)\, dx,
	\end{align}
	extends to a contraction $T_H:\tilde{A}_c(G,K)\to\tilde{A}_c(G/H,K/H)$. Moreover, if $(G,K)$ has relative property (T)${}_c$, and $P$ is the Kazhdan projection of $\tilde{A}_c(G,K)$, then $(G/H,K/H)$ has relative property (T)${}_c$, and $T_H(P)$ is the Kazhdan projection of $\tilde{A}_c(G/H,K/H)$.
\end{lem}
\begin{proof}
	Let $\pi\in\mathcal{R}_c(G/H,K/H)$, and let $\tilde{\pi}=\pi\circ q$, where $q:G\to G/H$ is the quotient map. Then $\tilde{\pi}$ is an element of $\mathcal{R}_c(G,K)$ because
	\begin{align*}
		\mathcal{H}_{\tilde{\pi}}^K=\mathcal{H}_{\pi}^{K/H}=\mathcal{H}_{\pi}^{G/H}=\mathcal{H}_{\tilde{\pi}}^G.
	\end{align*}
	Then, for every $f\in L^1(G)$, $\tilde{\pi}(f)=\pi(T_Hf)$. Hence
	\begin{align*}
		\|\pi(T_Hf)\|\leq \|\tilde{\pi}(f)\| \leq \|f\|_{\tilde{A}_c(G,K)}.
	\end{align*}
	Taking the supremum over $\mathcal{R}_c(G/H,K/H)$, we get
	\begin{align*}
		\|T_Hf\|_{\tilde{A}_c(G/H,K/H)} \leq \|f\|_{\tilde{A}_c(G,K)},
	\end{align*}
	and therefore $T_H$ extends to a contraction $\tilde{A}_c(G,K)\to\tilde{A}_c(G/H,K/H)$.	Finally, observing that the adjoint map $T_H^*:B_c(G/H,K/H)\to B_c(G,K)$ is given by $T_H^*(\varphi)=\varphi\circ q$, one easily checks that, if $P\in\tilde{A}_c(G,K)$ is the unique invariant mean on $B_c(G,K)$, then $T_H(P)$ is the unique invariant mean on $B_c(G/H,K/H)$.
\end{proof}

We will also need the following result, whose proof is very similar to that of Lemma \ref{Lem_T_H}.

\begin{lem}\label{Lem_S_K,H}
	Let $G$ be a locally compact group, $H, K$ two closed, normal subgroups of $G$ such that $H\leq K$, and $c\geq 1$. Then the identity map $L^1(G)\to L^1(G)$ extends to a contraction $S_{K,H}:\tilde{A}_c(G,K)\to\tilde{A}_c(G,H)$. Moreover, if $(G,K)$ has relative property (T)${}_c$, and $P$ is the Kazhdan projection of $\tilde{A}_c(G,K)$, then $(G,H)$ has relative property (T)${}_c$, and $S_{K,H}(P)$ is the Kazhdan projection of $\tilde{A}_c(G,H)$.
\end{lem}
\begin{proof}
	The result follows from observing that $\mathcal{R}_c(G,H)$ is contained in $\mathcal{R}_c(G,K)$. By the definitions of the norms on $\tilde{A}_c(G,K)$ and $\tilde{A}_c(G,H)$, $S_{K,H}:\tilde{A}_c(G,K)\to\tilde{A}_c(G,H)$ is a well-defined contraction. Moreover, its adjoint map is simply the inclusion $B_c(G,H)\hookrightarrow B_c(G,K)$. Hence, if $P\in\tilde{A}_c(G,K)$ is the unique invariant mean on $B_c(G,K)$, then $S_{K,H}(P)$ is the unique invariant mean on $B_c(G,H)$.
\end{proof}

Now we can prove Lemma \ref{Lem_(G/H,K/H)}.

\begin{proof}[Proof of Lemma \ref{Lem_(G/H,K/H)}]
	Assume first that $(G,K)$ has relative property (T)${}_c$. By Lemmas \ref{Lem_T_H} and \ref{Lem_S_K,H}, both $(G/H,K/H)$ and $(G,H)$ have relative property (T)${}_c$.\\
	Now assume that $(G/H,K/H)$ and $(G,H)$ have relative property (T)${}_c$. Let $\pi\in\mathcal{R}_c(G,K)$, and let $p_G$ and $p_H$ denote the projection onto $\mathcal{H}_\pi^G$ along $\left(\mathcal{H}_{\pi^*}^G\right)^\perp$ and the projection onto $\mathcal{H}_\pi^H$ along $\left(\mathcal{H}_{\pi^*}^H\right)^\perp$ respectively. Recall that, since $H$ is a normal subgroup, the spaces $\mathcal{H}_\pi^H$, $\left(\mathcal{H}_{\pi^*}^H\right)^\perp$ are also $G$-invariant.	From the decompositions \eqref{dec_G-inv} and \eqref{dec_H-inv}, we can write
	\begin{align*}
		\mathcal{H}_\pi &=p_G \mathcal{H}_\pi \oplus (1-p_G) \mathcal{H}_\pi\\
		&= p_H p_G \mathcal{H}_\pi \oplus (1-p_H)p_G \mathcal{H}_\pi \oplus p_H (1-p_G) \mathcal{H}_\pi \oplus (1-p_H)(1-p_G) \mathcal{H}_\pi.
	\end{align*}
	Moreover, since $\mathcal{H}_\pi^G\subseteq \mathcal{H}_\pi^H$, we have $p_Hp_G=p_G$ and $(1-p_H)(1-p_G)=(1-p_H)$. Thus we obtain the following decomposition into $G$-invariant subspaces:
	\begin{align*}
		\mathcal{H}_\pi=\mathcal{H}_1\oplus\mathcal{H}_2,
	\end{align*}
	where $\mathcal{H}_1=p_G\mathcal{H}_\pi \oplus (1-p_H)\mathcal{H}_{\pi}$ and $\mathcal{H}_2=p_H (1-p_G)\mathcal{H}_\pi$. Observe that the restriction of $\pi$ to $\mathcal{H}_1$ belongs to $\mathcal{R}_c(G,H)$. By Lemma \ref{Lem_char_p_pi}, the projections $p_1=p_G+(1-p_H):\mathcal{H}_\pi\to\mathcal{H}_1$ and $p_2=(p_H-p_G):\mathcal{H}_\pi\to\mathcal{H}_2$ satisfy
	\begin{align*}
		\|p_1\| &\leq 2c+1, & \|p_2\| &\leq 2c.
	\end{align*} 
	For every $f\in L^1(G)$ and $\xi\in\mathcal{H}_\pi$, we have
	\begin{align*}
		\pi(f)\xi &= \pi(f)p_1\xi + \pi(f)p_2\xi\\
		 &= \pi(f)p_1\xi + \tilde{\pi}(T_Hf)p_2\xi,
	\end{align*}
	where $\tilde{\pi}\in\mathcal{R}_c(G/H,K/H)$ is the representation on $\mathcal{H}_2$, given by $\tilde{\pi}(\dot{s})=\pi(s)$, and $T_H$ is as in \eqref{def_T_H}. Here we are using the fact that $\mathcal{H}_2\subseteq\mathcal{H}_\pi^H$. Then
	\begin{align*}
		\|\pi(f)\xi\| &\leq \|\pi(f)p_1\xi\| + \|\tilde{\pi}(T_Hf)p_2\xi\| \\
		&\leq \|f\|_{\tilde{A}_c(G,H)}\|p_1\xi\| + \|T_Hf\|_{\tilde{A}_c(G/H,K/H)}\|p_2\xi\| \\
		&\leq (2c+1)\|f\|_{\tilde{A}_c(G,H)}\|\xi\| + 2c\|T_Hf\|_{\tilde{A}_c(G/H,K/H)}\|\xi\|.
	\end{align*}
	Since this holds for any $\pi\in\mathcal{R}_c(G,K)$ and $\xi\in\mathcal{H}_\pi$, we obtain
	\begin{align}\label{est_norm_A_c}
		\|f\|_{\tilde{A}_c(G,K)} \leq (2c+1)\|f\|_{\tilde{A}_c(G,H)} + 2c\|T_Hf\|_{\tilde{A}_c(G/H,K/H)}
	\end{align}
	for all $f\in L^1(G)$. Let now $m$ be the unique invariant mean on $B_c(G,K)$. Extending $m$ to a state on $L^\infty(G)$, we may find a net of probability measures $(f_i)$ such that $f_i\to m$ in $\sigma(L^\infty(G)^*,L^\infty(G))$; see the proof \cite[Lemma 5.11]{HaKndL}. Since $m$ is also the unique invariant mean on $B_c(G,H)$, which we know belongs to $\tilde{A}_c(G,H)$, taking convex combinations, we find a sequence of continuous, compactly supported probability measures $(\rho_n)$ such that $\|\rho_n-m\|_{\tilde{A}_c(G,H)}\to 0$, and $\rho_n\to m$ in $\sigma(B_c(G,K)^*,B_c(G,K))$; see the proof of \cite[Lemma 4.2]{Ver}. Moreover, observing that $T_H^{**}m$ is the unique invariant mean on $B_c(G/H,K/H)$, by the same argument, we may assume that $\|T_H\rho_n-T_H m\|_{\tilde{A}_c(G/H,K/H)}\to 0$. These two facts, together with \eqref{est_norm_A_c}, show that $(\rho_n)$ is a Cauchy sequence in $\tilde{A}_c(G,K)\subseteq B_c(G,K)^*$, and therefore it converges in norm to its weak*-limit $m$. We conclude that $m$ belongs to $\tilde{A}_c(G,K)$, and thus $(G,K)$ has relative property (T)${}_c$.
\end{proof}

By setting $K=G$ in Lemma \ref{Lem_(G/H,K/H)}, we obtain Proposition \ref{Prop_G/H_(G,H)} as a particular case.

\section{Relative property \texorpdfstring{(T$_{ub}$)}{(Tub)} for semidirect products of an abelian group}
\label{sec:ab}

In this section, we prove Theorem \ref{thm:semidirectab}, which gives the equivalence between relative property (T) and relative property (T${}_{ub}$) for semidirect products with abelian groups.

\begin{proof}[Proof of Theorem \ref{thm:semidirectab}]
    The implication (T${}_{ub}$)$\Rightarrow$ (T) is immediate. Assume that $(G,A)$ has relative (T) but not relative (T${}_{ub}$). By \cite[Lemma 2.5.1]{Cor}, $A$ is $\sigma$-compact. We will construct a net of probability measures verifying $(P1'),(P2),(P3)$, thus obtaining a contradiction by Theorem \ref{thm:cortes}. We will denote by $h\cdot a$ the action of $H$ on $A$, that is to say $h\cdot a=hah^{-1}$, to distinguish from the multiplication in the group $G$.

    To construct this net, we consider $\rho$ a uniformly bounded representation of $G$ on $\mathcal{H}$, with almost invariant vectors and with $\mathcal{H}^A=\{0\}$ (which exists since $(G,A)$ does not have relative (T${}_{ub}$)). Let $c=|\rho|$, since $A$ is abelian and hence amenable, there exists an invertible operator $T\in \mathbf{B}(\mathcal{H})$ such that $\pi\vert_A$ is unitary, where $\pi(g)=T\circ \rho(g)\circ T^{-1}$, and with $\Vert T\Vert \Vert T^{-1}\Vert \leq c^2$. Thus, we may consider $E:\mathcal{B}(\hat{A})\to \mathbf{B}(\mathcal{H})$ the projection-valued measure associated to $\pi\vert_A$. For any $f$ bounded Borel function on $\hat{A}$, we obtain an operator $\pi(f)$. Note that $E(\{1\})=0$ by assumption, because $E(\{1\})$ is the projection onto $\mathcal{H}^A$.

    By Proposition \ref{prop:ubtoisom} and Remark \ref{rk:renorm}, consider an equivalent uniformly smooth norm $\Vert \cdot \Vert_\rho$ on $\mathcal{H}$ with respect to which $\rho$ is isometric, with norm ratio $\leq c$ and with modulus of smoothness at $r(\tau)\leq \frac{c^2\tau^2}{2}$. Let $\langle\cdot,\cdot\rangle$ denote the duality bracket for this norm (and not the inner product of $\mathcal{H}$) and $\xi\mapsto \xi^*$ the duality mapping. By assumption, for any compact subset $Q\subset G$ and for any $\delta>0$, there exists a vector $\xi_{Q,\delta}$ with $\Vert \xi_{Q,\delta}\Vert_\rho=1$ which is $(Q,\delta)$-invariant. Then by equivalence of norms and by Lemma \ref{lem:duality} and Remark \ref{rmk:contrag}, we have \begin{equation}\label{eq:inv}
        \underset{g\in Q}{\sup} \Vert \rho(g)\xi_{Q,\delta}-\xi_{Q,\delta}\Vert_\rho < c\delta
    \end{equation} and  \begin{equation}\label{eq:invdual}\underset{g\in Q}{\sup} \Vert \rho^*(g)\xi_{Q,\delta}^*-\xi_{Q,\delta}^*\Vert_{\rho^*} < 4c^3\delta.\end{equation}

    The map $f\mapsto \langle T^{-1}\pi(f)T\xi_{Q,\delta},\xi_{Q,\delta}^*\rangle$ is a linear functional on $C_0(\hat{A})$. Thus it is represented by a complex measure $\mu_{Q,\delta}\in M(\hat{A})$ such that \begin{equation}\label{eq:muQd}
        \int_{\hat{A}} f(\chi)d\mu_{Q,\delta}(\chi)=\langle T^{-1}\pi(f)T\xi,\xi^*\rangle.
    \end{equation}The equation \eqref{eq:muQd} extends to bounded Borel functions, hence $\mu_{Q,\delta}(\{1\})=0$. But these measures are not \textit{probability} measures. Thus, consider $$\Tilde{\mu}_{Q,\delta}=\frac{(\mu_{Q,\delta})_r^+}{(\mu_{Q,\delta})_r^+(\hat{A})}$$which is a positive probability measure (as defined in Section \ref{sec:measures}). Then we have $\Tilde{\mu}_{Q,\delta}(\{1\})=0$ so $(P2)$ holds.

    Let $B\in \mathcal{B}(\hat{A})$. Let $h\in H$, $a\in A$, by the action of $H$ on $A$, we have $\pi(h)\pi(a)\pi(h^{-1})=\pi(hah^{-1})=\pi(h\cdot a)$. Then by the construction of the projection-valued measure associated to $\pi$, we get $\pi(h)E(B)\pi(h^{-1})=E(hB)$. Thus we obtain \begin{equation}
        \forall h\in H,B\in \mathcal{B}(\hat{A}), \quad T\rho(h)T^{-1}E(Z)T\rho(h^{-1})T^{-1}=E(hB).
    \end{equation}
    Assume now that $h\in Q\cap H$, we have
    \begin{align*}
        \vert \mu_{Q,\delta}(h^{-1}B)- \mu_{Q,\delta}(B)\vert & = \vert  \langle T^{-1}E(h^{-1}B)T\xi_{Q,\delta},\xi_{Q,\delta}^*\rangle- \langle T^{-1}E(B)T\xi_{Q,\delta},\xi_{Q,\delta}^*\rangle\vert\\
        &=\vert\langle \rho(h^{-1})T^{-1}E(B)T\rho(h)\xi_{Q,\delta},\xi_{Q,\delta}^*\rangle- \langle T^{-1}E(B)T\xi_{Q,\delta},\xi_{Q,\delta}^*\rangle\vert\\
        &=\vert\langle T^{-1}E(B)T\rho(h)\xi_{Q,\delta},\rho^*(h)\xi_{Q,\delta}^*\rangle- \langle T^{-1}E(B)T\xi_{Q,\delta},\xi_{Q,\delta}^*\rangle\vert\\
        &\leq \vert\langle T^{-1}E(B)T\left(\rho(h)\xi_{Q,\delta}-\xi_{Q,\delta}\right),\rho^*(h)\xi_{Q,\delta}^*\rangle\vert\\
        &\phantom{\leq}+ \vert\langle T^{-1}E(B)T\xi_{Q,\delta},\rho^*(h)\xi_{Q,\delta}^*-\xi_{Q,\delta}^*\rangle\vert\\
        &\leq c^3\delta+4c^5\delta,
    \end{align*}
    the last inequality using \eqref{eq:inv} and \eqref{eq:invdual}. So by \eqref{eq:supbound}, we obtain \begin{equation*}
        \forall h\in Q^{-1}\cap H,\quad \Vert h_*(\mu_{Q,\delta})^+_r-(\mu_{Q,\delta})^+_r\Vert_{TV}<10c^5\delta.
    \end{equation*}
    Since $\mu_{Q,\delta}(\hat{A})=\langle \xi,\xi^*\rangle=1$, we have $\mu^+_{Q,\delta}(\hat{A})\geq 1$ thus
    \begin{equation}\label{eq:tvbound}
        \forall h\in Q^{-1}\cap H,\quad \Vert h_*\Tilde{\mu}_{Q,\delta}-\Tilde{\mu}_{Q,\delta}\Vert_{TV}<10c^5\delta.
    \end{equation}
    Thus $(P3)$ holds for this net of probability measures.

    Now let $a\in A$. Let $P$ be the support of the positive part $\mu_{Q,\delta}^+$. Then we have \begin{equation}\label{eq:postparttomu} \int_{\hat{A}} \vert 1-\chi(a)\vert^2 d\mu_{Q,\delta}^+(\chi)=\left\vert \int_{\hat{A}} 1_P(\chi)\vert 1-\chi(a)\vert^2 d\Re\mu_{Q,\delta}(\chi)\right\vert\leq \left\vert \int_{\hat{A}} 1_P(\chi)\vert 1-\chi(a)\vert^2 d\mu_{Q,\delta}(\chi)\right\vert.\end{equation}

    Let $\mathrm{ev}_a:\chi\mapsto \chi(a)$. Then \begin{align*}
        \left\vert \int_{\hat{A}} 1_P(\chi)\vert 1-\chi(a)\vert^2 d\mu_{Q,\delta}(\chi)\right\vert &= \left\vert \int_{\hat{A}} (1-\overline{\chi(a)})1_P(\chi)(1-\chi(a))\vert^2 d\mu_{Q,\delta}(\chi)\right\vert\\
        &=  \vert \langle T^{-1}\pi((1-\overline{\mathrm{ev}_a})1_P(1-\mathrm{ev}_a))T\xi_{Q,\delta},\xi_{Q,\delta}^*\rangle\vert\\
         &=  \vert \langle T^{-1}(\Id-\pi(a^{-1}))E(P)(\Id-\pi(a))T\xi_{Q,\delta},\xi_{Q,\delta}^*\rangle\vert\\
         &=  \vert \langle (\Id-\rho(a^{-1}))T^{-1}E(P)T(\Id-\rho(a))\xi_{Q,\delta},\xi_{Q,\delta}^*\rangle\vert\\
         &=  \vert \langle T^{-1}E(P)T(\xi_{Q,\delta}-\rho(a)\xi_{Q,\delta}),\xi_{Q,\delta}^*-\rho^*(a)\xi_{Q,\delta}^*\rangle\vert
    \end{align*}
    Now if $a\in Q\cap A$, we obtain \begin{equation}
        \label{eq:estimateA} \left\vert \int_{\hat{A}} 1_P(\chi)\vert 1-\chi(a)\vert^2 d\mu_{Q,\delta}(\chi)\right\vert \leq 4c^6\delta^2
    \end{equation}
    Finally, combining \eqref{eq:postparttomu}, \eqref{eq:estimateA}, Hölder's inequality and the fact that $\Tilde{\mu}_{Q,\delta}(\hat{A})\geq 1$, we get that \[\int_{\hat{A}} \chi(a)d\mu_{Q,\delta}(\chi)\to 1\]uniformly on compact subsets of $A$, hence $(P1')$.

    Thus we constructed a net of probability measures which verifies $(P1'),(P2),(P3)$. Since we assumed that the pair $(G,A)$ has relative (T), this is a contradiction, so $(G,A)$ has relative (T${}_{ub}$).
\end{proof}

\begin{rmk}
    If $A$ is a normal subgroup of $G$ but $G$ does not split as a semidirect product, set $H=G/A$. Then $(\neg P)$ implies relative property (T${}_{ub}$) with the same proof, but it is not clear whether (T) implies $(\neg P)$ in this case; see \cite[Theorem 7]{CorTes}.
\end{rmk}

Under some stronger hypothesis, we can obtain a better result. To do so, we need to define property (T) for triples (see, for example, \cite[Remark 0.2.1]{Jau}).

\begin{defi}
    Let $H,K$ be two closed subgroups of $G$. The triple $(G,H,K)$ has property (T)${}_c$ if for any $\pi\in \mathcal{R}_c(G)$ such that $1_H\prec\pi\vert_H$, one has $1_K \leq \pi\vert_K$.

    The triple $(G,H,K)$ has property (T${}_{ub}$) if $(G,H,K)$ has property (T)${}_c$ for any $c\geq 1$.
\end{defi}

When $G=H\ltimes K$ is a semidirect product, a variant of this property was introduced in \cite{Sha} (in the unitary setting) as ``strong property (T) of $(H,K)$'', see also \cite[Definition 5.2]{BFGM}. We prefer the terminology of property (T) of a triple, which is more general.

\begin{thm}
    Let $A$ be a locally compact abelian group and $G=H\ltimes A$ a locally compact group. Let $1$ be the trivial character of $A$. If the only $H$-invariant mean on the Pontryagin dual $\hat{A}$ is the Dirac mass $\delta_{1}$, then $(G,H,A)$ has property (T${}_{ub}$).
\end{thm}
\begin{proof}
    The proof is very similar but simpler: construct the same net of probability measures $(\Tilde{\mu}_{Q,\delta})$ (one can actually use the total variation instead of the positive part here). Notice that the assumption still allows us to conclude that $\Tilde{\mu}_{Q,\delta}(\{1\})=0$ and that \eqref{eq:tvbound} relies only on almost invariance for $\pi\vert_H$ and not on all of $G$. Then by Banach--Alaoglu, take a subnet of these probability measures converging in the weak-$*$ topology of $\mathcal{L}^\infty(\hat{A})$. The limit $\nu$ is an $H$-invariant mean, hence by assumption $\delta_1$, which contradicts $\Tilde{\mu}_{Q,\delta}(\{1\})=0$ for any $Q,\delta$.
\end{proof}

\begin{rmk}
    The characterisation of Theorem \ref{thm:cortes} extends beyond the realm of representations: a pair $(G,A)$ where $G$ is a semidirect product of the abelian group $A$ has relative property (T) if and only if it has relative property (TTT) of Ozawa \cite{Oza}, which deals with ``almost representations'' in a precise sense that we will not detail here. As mentioned in the introduction, relative property (T) is implicitly used in \cite{Kaz} to prove that higher rank simple groups have property (T). The same idea can be used to show that higher rank groups have property (TTT) \cite{Dum24, Oza}. The ideas developed in this section partially extend to this more general setting, see \cite[Chapter 6]{Dum25}.
\end{rmk}

\section{Relative property \texorpdfstring{(T$_{ub}$)}{(Tub)} for semidirect products of a nilpotent group}\label{sec:nil}

In this section, we extend Theorem \ref{thm:semidirectab} from abelian groups to nilpotent groups. Let $N$ be a locally compact group. Recall that $[N,N]$ denotes the commutator subgroup of $N$. The group $N$ is said to be nilpotent if it has a finite lower central series \[N=N_0\triangleright N_1 \triangleright \cdots \triangleright N_n=\{e\}\]with $N_{i+1}=\overline{[N_i,N]}$. The nilpotency class of $N$ is the smallest $n$ such that $N_n=\{e\}$. A group $G$ is nilpotent of class $1$ if and only if it is abelian.

\begin{rmk}
    This definition is equivalent to the usual \emph{algebraic} definition of nilpotent groups, i.e. without taking closure of commutator subgroups, since $N$ is assumed to be Hausdorff. Since we are working with locally compact groups, it is more convenient to consider closed subgroups in the definition, to be able to take quotients.
\end{rmk}

As mentioned in the introduction, we follow the ideas of \cite{CWMS} and begin by proving Theorem \ref{thm:quotab_intro}. This theorem should be seen as a twofold generalisation of a theorem of Serre, which may be of independent interest; see Corollary \ref{Cor_Serre_intro}. When $H=G$ and representations are unitary, the usual theorem of Serre \cite[Theorem 1.7.11]{BHV} is proved by considering an irreducible representation $\pi$ of $G$ and looking at $\pi\otimes \pi^*$. Since $A$ is central, Schur's lemma implies that $\pi\otimes \pi^*$ descends to $G/A$.

In general, the situation is more complicated. Theorem \ref{thm:quotab_intro} is proven in the case of unitary representations in \cite[Theorem 4.1]{CWMS}. We will adapt their proof to the case of uniformly bounded representations. To do so, we will need to introduce several Banach norms to obtain isometries at every step, making the proof technical. Before going into the details, let us sketch the steps of the proof.

\begin{enumerate}
    \item From $\pi\in \mathcal{R}_c(G)$, construct a representation $\pi\otimes_{\hat{A}} \pi^*$ of $G/A$. Unlike the easy case $G=H$, the tensor product will happen in fibres of a direct integral.
    \item Ensure that almost invariant vectors of $\pi$ yield almost invariant vectors of $\pi\otimes_{\hat{A}} \pi^*$; this is the content of Proposition \ref{prop:tensorinv}.
    \item Apply relative property (T)${}_c$ to obtain an $H$-invariant vector for $\pi\otimes_{\hat{A}} \pi^*$.
    \item Justify that we can recover from it an invariant \emph{line} in the original space, this is the content of Proposition \ref{prop:tensortorep}.
\end{enumerate}

Let $G$ be a locally compact second countable group and $A$ a closed abelian normal subgroup. Let $\pi$ be a uniformly bounded representation of $G$ on an Hilbert space $\mathcal{H}_\pi$. Since $A$ is unitarisable, we may assume as before that $\pi\vert_A$ is unitary.  

Our first goal is to construct a representation $\pi'$ of $G/A$ from $\pi$. When $A$ is central and $\pi$ irreducible, it suffices by Schur's lemma to take a tensor product. In general, the representation $\pi'$ will be a fibred tensor product, as constructed in \cite[Section 3]{CWMS}.

Again, since the restriction $\pi\vert_A$ is unitary, the representation theory of abelian groups provides a projection-valued measure $E$ on $\hat{A}$. Then the equality $\pi(g)\pi(a)\pi(g^{-1})=\pi(g \cdot a)$ implies that $\pi(g)E(B)\pi(g^{-1})=E(gB)$ for any $g\in G$ and $B$ Borel subset of $\hat{A}$. The pair $(\pi,E)$ forms a system of imprimitivity in the sense of \cite[Chapter VI]{Var}, with the \emph{caveat} that $\pi$ is not unitary.

The projection-valued measure $E$ is said to be homogeneous if there is a $\sigma$-finite measure $\mu$ on $\hat{A}$, a Hilbert space $\mathcal{H}$ and a unitary operator $W:\mathcal{H}_\pi\to L^2(\hat{A},\mu;\mathcal{H})$ such that $WE(B)W^{-1}f=1_{B}f$ for any $f\in L^2(\hat{A},\mu;\mathcal{H})$ and any Borel set $B$. Furthermore, in this case, the measure $\mu$ is quasi-invariant under $G$, i.e. $\mu$ and $g_*\mu$ are in the same measure class for all $g\in G$. 

If the pvm $E$ is homogeneous, by \cite[Theorem 6.11]{Var}\footnote{In Varadarajan, the representations considered are unitary. However, one can check that \cite[Lemmas 6.4 and 6.5]{Var} are stated for arbitrary bounded maps. Furthermore, \cite[Lemma 6.6]{Var} is stated for unitary representations, but extends to our setting as it only uses \cite[Thm 32C]{Loom}, which works for arbitrary bounded representations on reflexive Banach spaces. Notice however that we still need a projection-valued measure, which we only get by unitarising our representation on the abelian subgroup $A$.} there exists a Hilbert space $\mathcal{H}$, a cocycle $\alpha:G\times \hat{A}\to \mathbf{B}(\mathcal{H})$ and a quasi-invariant measure $\mu$ on $\hat{A}$ such that, up to isomorphism, $\pi$ is the representation on $L^2(\hat{A},\mu;\mathcal{H})$ given by
\begin{align}\label{pi_rep_L2(A,H)}
    (\pi(g)f)(\chi)=\sqrt{D(g,\chi)}\alpha(g,\chi)f(g^{-1}\chi),
\end{align}
where $D$ is the Radon--Nikodym cocycle associated to the action of $G$ on $(\hat{A},\mu)$. 

Let $\alpha^*(g,\chi)=\alpha(g^{-1},g^{-1}\chi)^*=(\alpha(g,\chi)^{-1})^*$ be the contragredient cocycle. We define a new representation $\pi'=\pi\otimes_{\hat{A}} \pi^*$ on $L^2(\hat{A},\mu;\mathcal{H}\otimes \mathcal{H}^*)$ by
\begin{align}\label{pi'_tensor_rep}
    (\pi'(g)f)(\chi)=\sqrt{D(g,\chi)}(\alpha(g,\chi)\otimes \alpha^*(g,\chi))f(g^{-1}\chi).
\end{align}
Then $\pi'$ is uniformly bounded, and $\vert \pi'\vert=\sup \Vert \pi'(g)\Vert\leq c^2$. The action of $A$ on $\hat{A}$ induced by conjugation is trivial, and the cocycle $\alpha(a,\chi)$ is the multiplication by the character $\chi(a)$, so that $\pi'\vert_A=\Id$.

\begin{rmk}
    It follows from the proof of \cite[Theorem 6.11]{Var} and from \cite[Lemma 6.5]{Var} that each $\alpha(g,\chi)$ is bounded by $c=|\pi|$. Arbitrary products of these operators need not be bounded by $c$ since, unlike representations, the image of a cocycle is not a group. Morover, the map $(S,T)\mapsto S\otimes T$ is SOT-continuous on bounded sets so the tensor cocycle is also Borel.
\end{rmk}

By Proposition \ref{prop:ubtoisom}, every uniformly bounded representation admits an equivalent (ucus) Banach norm with respect to which it is isometric. However, here we want to work in the fibre $\mathcal{H}$. For any $\chi\in \hat{A}$, define a new norm on $\mathcal{H}^*$ by $\Vert \phi\Vert_{\chi,*}=\underset{g\in G}{\sup} \Vert\alpha^*(g,g\chi)\phi\Vert$. Then $\Vert\cdot \Vert_{\chi,*}$ is equivalent to the original Hilbert norm, with $$\Vert \phi\Vert \leq \Vert\phi\Vert_{\chi,*}\leq c\Vert \phi\Vert.$$ Furthermore, the cocycle identity gives that $\alpha^*(g,x):(\mathcal{H}^*,\Vert\cdot\Vert_{g^{-1}\chi,*}) \to (\mathcal{H}^*,\Vert \cdot\Vert_{\chi,*})$ is an isometry. Let $\Vert \cdot \Vert_{\chi}$ be the predual norm on $\mathcal{H}$, then$$\frac{1}{c}\Vert \xi\Vert \leq \Vert \xi\Vert_\chi\leq \Vert \xi\Vert$$ and $\alpha(g,x):(\mathcal{H},\Vert\cdot\Vert_{g^{-1}\chi}) \to (\mathcal{H},\Vert \cdot\Vert_{\chi})$ is an isometry. Furthermore, the norms $\Vert \cdot\Vert_\chi$ are all uniformly smooth, with a modulus of smoothness at most $\sqrt{1+c^2\tau^2}-1\leq \frac{c^2\tau^2}{2}$. Let $J_\chi:(\mathcal{H},\Vert\cdot\Vert_\chi)\to (\mathcal{H}^*,\Vert\cdot\Vert_{\chi,*})$ denote the duality mapping. Then $$\alpha^*(g,\chi)J_{g^{-1}\chi}(\xi)=J_\chi(\alpha(g,\chi)\xi).$$
An operator $T\in \mathbf{B}(\mathcal{H},\Vert\cdot\Vert_\chi)$ is absolutely $2$-summing if there exists a constant $C>0$ such that for every finite family of vectors $x_1,\dots,x_n\in \mathcal{H}$, \[\left(\sum_{i=1}^n \Vert Tx_i\Vert_\chi^2\right)^{\frac{1}{2}}\leq C \underset{\phi\in \mathcal{H}^*,\Vert \phi \Vert_{\chi,*}\leq 1}{\sup} \left(\sum_{i=1}^n \vert\phi(x_i)\vert^2\right)^{\frac{1}{2}}.\]The absolutely $2$-summing norm of $T$, denoted $N_\chi(T)$, is the infimum of all such constants. By \cite[Theorem 4.10]{DieJarTon}, the space of absolutely 2-summing operators is exactly the space of Hilbert--Schmidt operators $HS(\mathcal{H})$ and $N_\chi$ defines a norm on $HS(\mathcal{H})\simeq\mathcal{H}\otimes \mathcal{H}^*$ equivalent to the Hilbert--Schmidt norm (with norm ratio at most $c^2$) and which is a cross norm, i.e. $N_\chi(\xi\otimes \phi)=\Vert \xi\Vert_\chi\Vert \phi\Vert_{\chi,*}$.

For any $\xi$, the map $\chi\mapsto \Vert \xi\Vert_\chi$ is measurable. Thus if $f\in L^2(\hat{A},\mu;\mathcal{H})$, we can define $$N(f)=\left(\int_{\hat{A}} \Vert f(\chi)\Vert_\chi^2 d\mu(\chi)\right)^{1/2}.$$ Then $N$ is a norm on $L^2(\hat{A},\mu;\mathcal{H})$, which is again equivalent to the natural Hilbert norm $\Vert\cdot\Vert_2$. Since the family of norms $(N_\chi)$ is equi-uniformly smooth, endowed with this norm, the space is uniformly smooth (mimic the proof of \cite[Theorem 1.e.9]{LinTza} using that all spaces have the same bound on the modulus of smoothness) with modulus of smoothness $r_N$. The dual space is $L^2(\hat{A},\mu;\mathcal{H}^*)$ with the dual norm $N_*$ such that $N_*(f)^2=\int_{\hat{A}} \Vert f(\chi)\Vert_{\chi,*}^2 d\mu(\chi)$. Then one can check that the duality mapping sends $f$ to $J(f):\chi\mapsto J_\chi(f(\chi))$.

We must adapt \cite[Proposition 3.6]{CWMS} to the context of uniformly bounded representations.
\begin{prop}\label{prop:tensorinv}
Let $G$ be a locally compact second countable group, $A$ be a closed abelian normal subgroup of $G$, and $c\geq 1$. Let $\pi\in\mathcal{R}_c(G)$ such that the associated pvm on $\hat{A}$ is homogeneous. Let $f\in L^2(\hat{A},\mu;\mathcal{H})$ be a $(Q,\varepsilon)$-invariant unit vector of $\pi$, under the identification \eqref{pi_rep_L2(A,H)}. Define $f'\in L^2(\hat{A},\mu;\mathcal{H}\otimes \mathcal{H}^*)$ by $$f'(\chi)=\frac{1}{\Vert f(\chi)\Vert_\chi} f(\chi)\otimes J_\chi(f(\chi)),$$where by convention $\frac{0}{0}=0$. Then $f''=f'/\Vert f'\Vert$ is a $(Q,\varepsilon')$-invariant unit vector of $\pi\otimes_{\hat{A}}\pi^*$, with $$\varepsilon'=2\sqrt{2}c\left(c\varepsilon+\frac{r_N(2c\varepsilon)}{\varepsilon}\right)\underset{\varepsilon \to 0}{\longrightarrow}0,$$
where $r_N$ is the modulus of smoothness of $(L^2(\hat{A},\mu;\mathcal{H}),N)$.
\end{prop}

Before the proof, we point out that \cite[Lemma 3.7]{CWMS} remains true when $\mathcal{H}_i$ are replaced by Banach spaces, as long as we take a cross-norm on the tensor product. 
\begin{lem}\label{lem:tensormaj}
   Let $(E_i,\Vert \cdot \Vert_i)$, $i=1,2$, be two Banach spaces and consider the algebraic tensor product $E_1\otimes E_2$ endowed with any cross-norm $\Vert \cdot \Vert$. Let $v_i,w_i\in E_i$, with $\Vert w_1\Vert_1=\Vert w_2\Vert_2$. For any vector $z_i\in E_i$, denote $\hat{z}_i$ a unit vector such that $z_i=\Vert z_i\Vert_i\hat{z}_i$. Then for any $D\geq 0$, $$\Vert Dw_1 \otimes \hat{w}_2-\hat{v}_1\otimes v_2\Vert\leq 2\Vert D w_1-v_1 \Vert_1 + \Vert Dw_ 2-v_2\Vert_2$$ 
\end{lem}
\begin{proof}
    The proof is identical to that of \cite[Lemma 3.7]{CWMS}.
\end{proof}

\begin{proof}[Proof of Proposition \ref{prop:tensorinv}]
    Since $f$ is a unit vector and $\Vert\cdot \Vert_\chi$ and $\Vert\cdot \Vert$ have norm ratio at most $c$, $\Vert f'\Vert_2\in [\frac{1}{c},1]$ so it suffices to control the invariance of $f'$.

    Let $g\in Q$. Apply Lemma \ref{lem:tensormaj} to the spaces $(\mathcal{H},\Vert\cdot\Vert_\chi)$ and $(\mathcal{H}^*,\Vert \cdot\Vert_{\chi_*})$, with $D=\sqrt{D(g,\chi)}$, $v_1=f(\chi)$, $w_1=\alpha(g,\chi)f(g^{-1}\chi)$, $v_2=J_\chi(v_1)$, $w_2=J_\chi(w_1)$. Since $$\alpha^*(g,\chi)J_{g^{-1}\chi}(f(g^{-1}\chi))=J_{\chi}(\alpha(g,\chi)f(g^{-1}\chi),$$we have $(\pi'(g)f')(\chi)=Dw_1\otimes \hat{w}_2$. 

    Then \begin{align*}
        N_\chi((\pi'(g)f')(\chi)-f'(\chi)) &\leq 2\Vert Dw_1-v_1\Vert_\chi + \Vert Dw_2-v_2\Vert_{\chi,*}\\
        &\leq 2\Vert (\pi(g)f)(\chi)-f(\chi)\Vert_\chi+\Vert J_\chi((\pi(g)f)(\chi))-J_\chi(f(\chi))\Vert_{\chi,*}
    \end{align*}

By equivalence of norms, $$\Vert (\pi'(g)f')(\chi)-f'(\chi)\Vert\leq 2c \Vert (\pi(g)f)(\chi)-f(\chi)\Vert+c\Vert J_\chi((\pi(g)f)(\chi))-J_\chi(f(\chi))\Vert_{\chi,*}$$so $$\Vert (\pi'(g)f')(\chi)-f'(\chi)\Vert^2\leq 8c^2 \Vert (\pi(g)f)(\chi)-f(\chi)\Vert^2+2c^2\Vert J_\chi((\pi(g)f)(\chi))-J_\chi(f(\chi))\Vert_{\chi,*}.$$Thus integrating, $$\Vert \pi'(g)f'-f'\Vert_2^2 \leq 8c^2\Vert \pi(g)f-f\Vert_2^2+2c^2 N_*(J(\pi(g)f)-J(f)).$$

By assumption, since $g\in Q$, $\Vert \pi(g)f-f\Vert_2<\varepsilon$. Thus by Lemma \ref{lem:duality}, $$N_*(J(\pi(g)f)-J(f)) \leq 2\frac{r_N(2c\varepsilon)}{c\varepsilon}.$$

Finally, we obtain \begin{equation*}
    \Vert \pi'(g)f'-f'\Vert_2  \leq \sqrt{8c^2\varepsilon^2+8c^2 \frac{r_N(2c\varepsilon)^2}{c^2\varepsilon^2}} \leq 2\sqrt{2}\left(c\varepsilon + \frac{r_N(2c\varepsilon)}{\varepsilon}\right)
\end{equation*}which goes to $0$ as $\varepsilon\to 0$ by uniform smoothness of $N$.
\end{proof}

Next, we prove a version of \cite[Proposition 3.8]{CWMS} for uniformly bounded representations.
\begin{prop}\label{prop:tensortorep} Let $G$ be a locally compact second countable group and let $\rho$ be a uniformly bounded representation of $G$ on $\mathcal{H}$. Let $\xi\in \mathcal{H}$ and let $\xi^*\in \mathcal{H}^*$ be a linear form such that $\xi^*(\xi)=\Vert \xi\Vert^2$. Let $\eta'\in \mathcal{H}\otimes \mathcal{H}^*$ be a $G$-invariant vector of $\rho\otimes \rho^*$. Then there exists a vector $\eta\in \mathcal{H}$ such that the line $\C\eta$ is $\rho(G)$-invariant and such that $$\Vert \eta-\xi\Vert \Vert \xi\Vert \leq 7\Vert \eta'-\xi\otimes \xi^*\Vert.$$ 
\end{prop}

\begin{proof}
    We may assume that $\Vert \xi\Vert=1$. We may also assume that $\delta=\Vert \eta'-\xi\otimes \xi^*\Vert<\frac{1}{7}$, otherwise the inequality we want to prove holds with $\eta=0$. Under the identification of $\mathcal{H}\otimes \mathcal{H}^*$ with the space $HS(\mathcal{H})$ of Hilbert--Schmidt operators, the representation of $G$ is $(g,S)\mapsto \rho(g)S\rho(g^{-1})$. To avoid confusion with the numerous other norms, we denote the Hilbert--Schmidt norm of $S$ by $\Vert S\Vert_{HS}$ and the operator norm $\Vert S\Vert_{\op}$. Let $T$ be the operator corresponding to $\eta'$. Then $T$ is an intertwining operator for $\rho$. Let $P_\xi$ be the operator corresponding to $\xi\otimes\xi^*$, then since $\xi^*(\xi)=1$, this is a projection onto $\C\xi$.

    The spectrum of $P_\xi$ is $\{0,1\}$. Let $\mathcal{C}$ be the circle of radius $1/2$ around $1$. Then for any $z\in \mathcal{C}$, we can consider the resolvent $R(z,P_\xi)=(z \Id-P_\xi)^{-1}$. For $z\in \mathcal{C}$, $\Vert R(z,P_\xi)\Vert_{\op}=2$. Since $$z\Id-T=(z\Id-P_\xi)(\Id-R(z,P_\xi)(T-P_\xi))$$and $$\Vert R(z,P_\xi)(T-P_\xi)\Vert_{\op} \leq \Vert R(z,P_\xi)\Vert_{\op} \Vert T-P_\xi\Vert_{\op}\leq 2\Vert T-P_\xi\Vert_{HS}\leq 2\delta<1,$$ $z\Id-T$ is invertible for $z\in\mathcal{C}$ and $$R(z,T)=\left(\sum_{n=0}^\infty (R(z,P_\xi)(T-P_\xi))^n\right) R(z,P_\xi)$$

    Then $\Vert R(z,T)-R(z,P_\xi)\Vert_{\op} \leq 2\sum_{n=1}^\infty (2\delta)^n=\frac{4\delta}{1-2\delta}\leq \frac{4}{5}$. Then consider the Riesz projection $$P_T=\frac{1}{2i\pi}\int_\mathcal{C} R(z,T)\ dz.$$Since $P_\xi=\frac{1}{2i\pi}\int_\mathcal{C} R(z,P_\xi)\ dz$, we have $$\Vert P_T-P_\xi\Vert_{\op}\leq \frac{1}{2\pi}\int_\mathcal{C}\Vert R(z,T)-R(z,P_\xi)\Vert_{\op}\ dz\leq \frac{1}{2\pi}\pi \frac{4\delta}{1-2\delta}\leq\frac{2}{5}<1.$$So the projection $P_T$ has the same rank as $P_\xi$, which is $1$. Let $\eta=P_T\xi$. Since $\rho(g)T=T\rho(g)$ for any $g$, then $\rho(g)P_T=P_T\rho(g)$ so the line $\C\eta=\mathrm{Im}P_T$ is $\rho(G)$-invariant. Thus, the vector $\eta$ is $\rho(\overline{[G,G]})$-invariant. Furthermore,
    \begin{align*}
        \Vert \eta-\xi\Vert&=\Vert P_T\xi-P_\xi\xi\Vert \leq \Vert P_T-P_\xi\Vert_{\op} \leq \frac{2\delta}{1-2\delta}\leq 7\delta.\qedhere
    \end{align*}
\end{proof}

Now we have all the ingredients for proving Theorem \ref{thm:quotab_intro}.

\begin{proof}[Proof of Theorem \ref{thm:quotab_intro}]
    Since the pair $(G/A,H/A)$ has relative property (T)${}_{c^2}$, we know that for any $\delta>0$, there is a compact subset $Q$ and $\varepsilon>0$ such that if $\xi$ is a $(Q,\varepsilon)$-invariant vector of $\pi\in \mathcal{R}_{c^2}(G/A)$, then there is $\eta$ which is $H/A$-invariant with $\Vert \eta-\xi\Vert <\delta/7$.

    Let $\pi\in \mathcal{R}_c(G)$ which has almost invariant vectors. Since $A$ is abelian, we may assume that $\pi\vert_A$ is unitary. Replacing $\pi$ by an infinite direct sum of copies of $\pi$, the associated pvm is homogeneous. Thus, we can construct $\pi'=\pi\otimes_{\hat{A}} \pi^*$ as in \eqref{pi'_tensor_rep}, which is also a representation of $G/A$. Fix $\varepsilon_0$ such that $$2\sqrt{2}c\left(c\varepsilon_0+\frac{r_N(2c\varepsilon_0)}{\varepsilon_0}\right)\leq \varepsilon,$$which exists by uniform smoothness of $(L^2(\hat{A},\mu;\mathcal{H}),N)$. By assumption, $\pi$ almost has invariant vectors so there exists a unit vector $f\in L^2(\hat{A},\mu;\mathcal{H})$ which is $(Q,\varepsilon_0)$-invariant. By Proposition \ref{prop:tensorinv}, we can construct a unit vector $f''\in L^2(\hat{A},\mu;\mathcal{H}\otimes \mathcal{H}^*)$ which is $(Q,\varepsilon)$-invariant under $\pi'$. Furthermore, there is $M\in[1, c]$ such that for any $\chi\in \hat{A}$, $$f''(\chi)=\frac{M}{\Vert f(\chi)\Vert_\chi} f(\chi)\otimes J_\chi(f(\chi)).$$Then, by relative property (T)${}_{c^2}$ for the pair $(G/A,H/A)$, there is an actual $\pi'(H)$-invariant vector $f'_H$ with $$\Vert f'_H-f''\Vert <\delta/7.$$

    So for all $h\in H$ and almost every $\chi\in \hat{A}$, $f'_H(\chi)=\sqrt{D(h,\chi)} \alpha'(h,\chi)f'_H(h^{-1}\chi)$. Now $\alpha'(h,\chi)=\alpha(h,\chi)\otimes \alpha^*(h,\chi)$ is an isometry $(\mathcal{H}\otimes \mathcal{H^*},N_{h^{-1}\chi})\to (\mathcal{H}\otimes \mathcal{H^*},N_{\chi})$, so $$N_\chi(f'_H(\chi))=\sqrt{D(h,\chi)} N_{h^{-1}\chi}(f'_H(h^{-1}\chi))$$ and the measure $N_\cdot(f'_H)\mu$ is $H$-invariant. Since $f'_H$ is square-integrable for the usual Hilbert norm on $\mathcal{H}\otimes \mathcal{H}^*$, which is (uniformly on $\chi$) equivalent to $N_\chi$, this measure is also finite. Thus by assumption, its support is contained in $\hat{A}^H$, the set of fixed points under the action of $H$. So up to changing $f'_H$ on a set of measure $0$, we may assume that $f'_H$ is zero outside of $\hat{A}^H$. Fix $\chi\in \hat{A}^H$. Then the map $h\mapsto \alpha(h,\chi)$ is a (uniformly bounded) representation $\rho_\chi$ of $H$ (using cocycle property and that $\chi$ is fixed). Furthermore, the Radon--Nikodym cocycle $D(h,\chi)$ is trivial for $h\in H$, $\chi\in \hat{A}^H$. Thus, $f'_H(\chi)\in \mathcal{H}\otimes \mathcal{H}^*$ is invariant under $\rho_\chi\otimes \rho_\chi^*$. By Proposition \ref{prop:tensortorep} applied with $\xi=f(\chi)$, $\xi^*=J_\chi(f(\chi))$ and $\eta'=\frac{f'_H(\chi)\Vert f(\chi)\Vert_\chi}{M}$, there is $\eta(\chi)\in \mathcal{H}$ which is $\rho_\chi(\overline{[H,H]})$-invariant and such that $$\Vert \eta(\chi)-f(\chi)\Vert \Vert f(\chi)\Vert \leq 7\Vert \eta'-f(\chi)\otimes J_\chi(f(\chi))\Vert.$$ Thus $$\Vert \eta(\chi)-f(\chi)\Vert\leq 7\frac{\Vert f(\chi)\Vert_\chi}{M\Vert f(\chi)\Vert }\Vert f'_H(\chi)-f''(\chi)\Vert=7\Vert f'_H(\chi)-f''(\chi)\Vert$$Furthermore, applying the Jankov--von Neumann uniformisation theorem \cite[Theorem 18.1]{Kec} to the analytic set of all pairs $(\chi,\eta(\chi))$ with $\eta(\chi)$ as above,  $\eta$ can be chosen to be a measurable function. Then $$\Vert \eta-f\Vert_2\leq 7\Vert f'_H-f'\Vert_2<\delta.$$Thus since $f$ is a unit vector, it suffices to take $\delta$ small enough, then $\eta\neq 0$ and is an $\overline{[H,H]}$-invariant vector in $L^2(\hat{A},\mu;\mathcal{H})$.
\end{proof}

\begin{rmk}\label{rmk:central}
    If we assume that $A$ is a subgroup of the centre $Z(H)$, then $H$ acts trivially on $A$, thus the assumption on the support of measures is automatically verified. 
\end{rmk}

As a first consequence, taking $H=G$, we immediately obtain Corollary \ref{Cor_Serre_intro}.

\begin{proof}[Proof of Corollary \ref{Cor_Serre_intro}]
    Since we are assuming that $(G/A,G/A)$ satisfies relative property (T)${}_{c^2}$, Theorem \ref{thm:quotab_intro} ensures that the pair $(G,G^{ab})$ has relative property (T)${}_c$. Furthermore, $G^{ab}$ satisfies property (T)${}_c$ because it is compact. By Proposition \ref{Prop_G/H_(G,H)}, we conclude that $G$ has property (T)${}_c$.
\end{proof}

A second consequence allows us to extend our result on abelian groups to nilpotent groups.
\begin{thm}\label{thm:nilpotent}
    Let $N$ be a closed nilpotent normal subgroup of $G$ locally compact second countable. Then $(G,N)$ has relative property (T${}_{ub}$) if and only if $(G/\overline{[N,N]},N^{ab})$ does.
\end{thm}

\begin{proof}
    If $(G,N)$ has relative property (T${}_{ub}$), then so does the quotient pair $(G/\overline{[N,N]},N^{ab})$ by Lemma \ref{Lem_(G/H,K/H)}.
    
    For the other direction, we assume that $(G/\overline{[N,N]},N^{ab})$ has property (T${}_{ub}$) and proceed by induction on the nilpotency class $n$ of $N$.

    If $n=1$, $[N,N]$ is trivial and $N$ is abelian so there is nothing to prove.

    Assume that the theorem holds for nilpotency class $k\leq n$ and that $N$ is nilpotent of class $n+1$. Consider the lower central series \[N=N_0\triangleright N_1\triangleright \cdots \triangleright N_{n+1}=\{e\}\]where $N_{i+1}=\overline{[N_i,N]}$. Then $N_n\leq Z(N)$. Set $N'=N/N_n$. Since $N_n$ is central and contained in all $N_i$'s, one can check that $$\overline{[N_i/N_n,N/N_n]}=\overline{[N_i,N]}/N_n=N_{i+1}/N_n$$ so that $N'$ is nilpotent of class $n$ with lower central series given by the $N_i/N_n$'s. Since $N_n\leq N$, $N_n\triangleleft G$ and by the third isomorphism theorem, we have both \[(N')^{ab}=N'/(N_1/N_n)\simeq N/N_1=N^{ab}\]and\[(G/N_n)/\overline{[N',N']}\simeq G/\overline{[N,N]}.\]

    Thus, by assumption, the pair $\left((G/N_n)/\overline{[N',N']},(N')^{ab}\right)$ has relative property (T${}_{ub}$) and by induction hypothesis, so does $(G/N_n,N/N_n)$. Now since $N_n$ is central in $N$, by Remark \ref{rmk:central} we can apply Theorem \ref{thm:quotab_intro} and obtain that $(G,\overline{[N,N]})$ has relative property (T${}_{ub}$). Hence, by Lemma \ref{Lem_(G/H,K/H)}, since both $(G,\overline{[N,N]})$ and $(G/\overline{[N,N]},N/\overline{[N,N]})$ have relative property (T)${}_c$ for all $c\geq 1$, the same holds for $(G,N)$.
\end{proof}

\begin{proof}[Proof of Theorem \ref{thm:mainthmnil}]
    By definition, if $(G,N)$ has relative property (T${}_{ub}$), then in particular $(G,N)$ has relative property (T). 
    
    Assume that $(G,N)$ has relative property (T). Then by Lemma \ref{Lem_(G/H,K/H)}, so does the quotient pair $(G/\overline{[N,N]},N^{ab})$. By Theorem \ref{thm:semidirectab}, $(G/\overline{[N,N]},N^{ab})$ has relative property (T${}_{ub}$). Thus by Theorem \ref{thm:nilpotent}, $(G,N)$ has relative property (T${}_{ub}$).
\end{proof}

\bibliographystyle{plain} 

\bibliography{Bibliography}

\end{document}